\documentclass{amsart}
\usepackage{amsmath, amsthm, amssymb, graphicx, tikz, zref} 
\usetikzlibrary{arrows}
\usepackage{enumitem}
\newtheorem{theorem}{Theorem}[section]

\newtheorem{lemma}[theorem]{Lemma}
\newtheorem{example}[theorem]{Example}
\newtheorem{proposition}[theorem]{Proposition}
\newtheorem{exercise}[theorem]{Exercise}

\newtheorem{definition}[theorem]{Definition}
\newtheorem{remark}[theorem]{Remark}

\usepackage[margin=1in]{geometry} 

\usepackage{placeins}
\newtheorem{defn}[theorem]{Definition}
\newtheorem{ex}[theorem]{Example}
\newtheorem{exc}[theorem]{Exercise}

\usepackage{makecell}

\usepackage[bb=libus]{mathalpha}
\newcommand*\circled[1]{\tikz[baseline=(char.base)]{
            \node[shape=circle,draw,inner sep=2pt] (char) {#1};}}

\usepackage{pdfpages}
\usepackage{hyperref}
\usepackage{faktor}

\usepackage[backend=biber, 
style=alphabetic,
sorting=ynt
]{biblatex}
\usepackage{quiver} 
\usepackage{tikz-cd}

\usetikzlibrary{decorations.markings, arrows.meta, calc} 

\usepackage{mathabx,epsfig}

\newcommand{\sslash}{\mathbin{/\mkern-6mu/}}

\hypersetup{
    colorlinks,
    linkcolor={red!50!black},
    citecolor={blue!50!black},
    urlcolor={blue!80!black}
}
\newcommand{\R}{\mathbb{R}}

\newcommand{\PP}{\mathbb{P}}\newcommand{\A}{\mathbb{A}}\newcommand{\Z}{\mathbb{Z}}\newcommand{\C}{\mathbb{C}}\newcommand{\N}{\mathbb{N}}\newcommand{\T}{\mathbb{T}}
\newcommand{\Q}{\mathbb{Q}}

\usepackage{verbatim}

\title{Enumerative Geometry and Tree-Level Gromov--Witten Invariants}
\author{Reginald Anderson}
\author{Carrie Frizzell}
\email{reginala@uci.edu}
\email{carrie.r.frizzell@gmail.com}
\date{\today}

\begin{document}

\begin{abstract} Here we review background in differential topology related to the calculation of an euler characteristic, and background on localization in equivariant cohomology. We then outline Gromov-Witten invariants in algebraic geometry and give examples of the genus 0 Gromov-Witten potential for $\PP^1, \PP^2$, and a genus $g>0$ Riemann surface. Kontsevich-Manin's recursive formula for $N_d$, the number of degree $d$ rational curves through $3d-1$ points in general position on $\PP^2$ is recovered. 

\end{abstract} 


\maketitle

\tableofcontents

\section{Introduction}

Through any two distinct points of $\PP^2(\C)$, there exists a unique line. Similarly, through any 5 points in sufficiently general position in $\PP^2$, there exists a unique conic through these 5 points. More generally, we can ask for $N_d$, the number of degree $d$ rational curves in $\PP^2$ through $3d-1$ points. Here, $3d-1$ is the ``right" number of points to consider, since this is the dimension of degree $d$ rational curves of $\PP^2$. We can see the dimension of the space of degree $d$ rational curves to $\PP^2$ by calculating the dimension of a dense open subset, $U$. A degree $d$ rational map in $U$ is given by a 3-tuple of degree $d$ polynomials in $u$ and $v$, for $[u:v]$ homogeneous coordinates on $\PP^1:$

\begin{align*}
    \left[ \sum_{j=0}^d a_j u^j v^{d-j} : \sum_{j=0}^d b_j u^j v^{d-j} : \sum_{j=0}^d c_j u^j v^{d-j} \right]. 
\end{align*}

Here there are $3d+3$ parameters from $\{a_j, b_j, c_j \text{ }|\text{ }0 \leq j \leq d \}$ and we subtract $1$ from homogeneous scaling by $\C^*$ on $\PP^2$. Lastly, we subtract $3$ from the dimension of $PGL(2,\C)$ acting on $\PP^1$ to count distinct curves without automorphisms (so-called ``stable maps"):

\begin{align*}
    \dim U &= 3(d+1) - 1 - 3\\
    &= 3d-1.
\end{align*}

It turns out that for cubics, the number of degree $3$ rational curves through 8 points in sufficiently general position in $\PP^2$ is no longer unique: $N_3=12$. Recursively, these numbers $N_d$ are given by Kontsevich-Manin's \cite{KontsevichManin_1994} formula 

\[ N_d = \sum_{\makecell{d_1+d_2=d, \\ d_1, d_2>0}} N_{d_1} N_{d_2} \left( d_1^2 d_2^2 \binom{3d-4}{3d_1-2} - d_1^3 d_2 \binom{3d-4}{3d_1-1}\right) \] 

 which considers those degree $d$ rational curves which are given by \textbf{stable maps} from $\PP^1$ to $\PP^2$ starting with the data $N_1=1$. The numbers $N_d$ are examples of Gromov-Witten invariants, and can be computed by 

\[ N_d = \int_{\overline{\mathcal{M}}_{0,3d-1}(\PP^2,d)} ev_1^*(P) \cup ev_2^*(P) \cup \cdots \cup ev_{3d-1}^*(P) \] where $P$ is Poincar\'{e} dual to the point class \cite{MSTextbook}. The topic of Gromov-Witten invariants makes contact with various branches of math and physics, such as partition functions in quantum field theories \cite{MSTextbook} and mirror symmetry \cite{coxkatxMSAG}. Before the recursive formula above was discovered, $N_d$ was unknown for $d\geq 5$ \cite{katzenumerative}. This formula follows from associativity in the big quantum cohomology ring $QH^*(\PP^2)$, using the notion of stable degree $d$ rational maps from $\PP^1 \rightarrow \PP^2$, where we have quotiented out by automorphisms of the image curve in $\PP^2$ \cite{kock2006invitation}. More generally, we can ask for the number $N_d$ of rational curves inside of a space $X$ of a given degree $d$ (where $d$ need not be an integer, but instead a homology class) via the notion of genus $0$ stable maps to $X$ whose image has a specified homology class $\beta$ in $H_2(X,\Z)$. Again, these Gromov-Witten invariants become integrals of equivariant cohomology classes over a compactification of a moduli space of genus $0$ stable maps with $n$ marked points. 

It is a classical fact that a nonsingular cubic surface contains exactly 27 lines. However, the number of conics contained in a cubic surface is infinite. To have a finite number of degree $d$ curves for all non-negative integers $d$, Clemens' conjecture suggests that we consider a generic quintic threefold 
\[Q = V(g), \hspace{1cm} g \in \Gamma( \mathcal{O}_{\PP^4}(3)) \]

For the quintic threefold, $N_1=2875$ was found by Schubert in the 19th century, $N_2=609250$ was found by Katz in 1985 \cite{katzenumerative}, and $N_3$ was correctly discovered by physicists Candelas-de la Ossa-Green-Parkes \cite{CANDELAS199121} in 1991, who found a formula for the number of rational degree $k$ curves, which they denoted $n_k$ and which are related to the $N_k$ above in a subtle way. This came from the sum over instantons via Equation (5.13) in \cite{CANDELAS199121}

\[ 5 + \sum_{k=1}^\infty \frac{n_k k^3 e^{2\pi i k t}}{1 - e^{2\pi i k t}} = 5 + n_1 e^{2\pi i k t} + (2^3n_2 + n_1) e^{4\pi i k t} + \cdots \] with $n_1 = 2875, n_2 = 609250,$ and so on. Givental proved this formula in \cite{givental1998mirrorformulaquinticthreefolds}, and one can ask for similar formulas in more general spaces. Due to the combinatorial description of the domain of a genus $0$ stable map, genus $0$ Gromov-Witten invariants are often referred to as ``tree-level." It should be emphasized that this exposition is in not meant to replace the many excellent surveys and works such as Kontsevich-Manin \cite{KontsevichManin_1994}, Behrend-Manin \cite{behrend1995stacksstablemapsgromovwitten}, Manin \cite{maninfrobenius}, and Fulton-Pandharipande \cite{fulton1997notesstablemapsquantum}. Rather, this covers enough background to elucidate some historically significant examples, and is focused on empowering the reader to work examples. Due to the strength of localization arguments in equivariant cohomology, tree-level Gromov-Witten invariants can be computed as sum over trees from torus-fixed loci in the presence of a holomorphic action of $\mathbb{T}$ on $X$, where $\mathbb{T}$ is a finite product of copies of $S^1$ and $\C^*$. As a warm-up for computing Gromov-Witten invariants, we recall some classical calculations of enumerative invariants of well-known spaces.

\section{Warm-up and Review}
\subsection{Chern class argument for 27 lines on non-singular cubic surface} 

\label{chernclassCubicSurface}

As a smooth degree $d$ hypersurface of $\PP^d$ with $\omega_X \cong \mathcal{O}_X(-1)$, cubic surfaces 
 \[ X = V(f), \hspace{1cm} f\in \Gamma( \mathcal{O}_{\PP^3}(3)) \] (i.e., $d=3$) are \textbf{del Pezzo}. A cubic surface $X$ can be constructed as the blow-up of $\PP^2$ in 6 points in sufficiently general position, which leads to a proof of the fact that $X$ contains 27 lines by considering the Picard group of $X$  \cite{hartshorne.alg.geom}. A proof using classical methods is given in \cite{reid2009undergraduate}. For the sake of exposition, here we recall a well-known proof of the fact that a cubic surface contains exactly 27 lines by giving a chern class argument as outlined in \cite{eisenbud20163264}. \\

\begin{remark}
The total chern class \[ c(\mathcal{E}) = 1 + c_1(\mathcal{E}) + c_2(\mathcal{E}) + \cdots \] of a vector bundle $\mathcal{E}$ on a quasi-projective variety $X$ is an algebraic invariant of the Chow ring $A(X)$ such that the following conditions hold:\\
\begin{enumerate}
    \item for a line bundle $\mathcal{L}$, \[ c(\mathcal{L}) = 1 + c_1(\mathcal{L})\] where $c_1(\mathcal{L}) \in A^1(X)$ is the class of the divisor of zeros minus the divisor of poles of a rational section of $\mathcal{L}$ \\
    \item Total chern classes are multiplicative on short exact sequences: for\\
    \[ 0 \rightarrow \mathcal{F} \rightarrow \mathcal{E} \rightarrow \mathcal{G} \rightarrow 0 \] a short exact sequence of vector bundles on $X$, we have \[ c(\mathcal{E}) = c(\mathcal{F})c(\mathcal{G}).\] 
    \item Total chern classes pull back: Given $\phi: Y\rightarrow X$ a morphism of nonsingular varieties, then \[ \phi^*(c(\mathcal{E})) = c(\phi^*\mathcal{E}).\] 
    \item We normalize the top chern class by setting the top chern class of the tangent bundle of $X$ equal to the euler class of the tangent bundle: \[ c_{top}(TX) = e(TX) \]

\end{enumerate}
\end{remark} 

\begin{proof} 

From defining $\mathbb{G}r(1,\PP^3)$ as the space of lines in $\PP^3$, projectivizing gives an isomorphism with the space of $2$-dimensional subspaces of $\A^4$: \[\mathbb{G}r(1, \PP^3) \cong \text{Gr}(2,\A^4). \] As such, $\mathbb{G}r(1,\PP^3)$ carries a tautological rank $2$ vector bundle $\mathcal{S}$, which associates to any line $\Lambda \in \mathbb{G}r(1,\PP^3)$ its associated 2-plane $Q\in Gr(2,\A^4)$. We can see that $\dim(\text{Gr}(2,\A^4))=4$ by taking an open set $U$ of a given $2$-plane $Q$ to be all $2$ planes $Q'$ such that $Q'$ surjects onto $Q$ via the first projection map $$\A^4 \cong Q \oplus Q^\perp \rightarrow Q.$$ That is, \[ U = \{ Q' \text{ }|\text{ } Q' \cap Q^\perp = \{0\} \}. \] Any such $2$-plane $Q'$ can be viewed as a unique map $Hom_{k-\text{v.s.}}(Q,Q^\perp)$ and represented by a $2\times 2$ matrix. In fact, this also describes the tangent bundle $T \text{Gr}(2,\A^4)$ as $Hom_{\text{v. bdl}}(\mathcal{S}, \mathcal{S}^\perp)$, so we have also found that $\dim ( \text{Gr}(2, \A^4))= \dim T\text{Gr}(2,\A^4) = 4$, using the fact that $\text{Gr}(2,\A^4)$ is smooth. A construction of the Chow ring $A(\mathbb{G}r(1,\PP^3))$ describing the intersection theory of subvarieties of $\mathbb{G}r(1,\PP^3)$ using Schubert classes is given in \cite{eisenbud20163264}.  

 Let $F_1(X) \subset \mathbb{G}r(1,\PP^3)$ denote the set of lines in $\PP^3$ which are contained in $X$. By Bezout's theorem, a line (which has degree $1$ and dimension $1$ in $\PP^3$) which is not contained in the degree $3$ surface $X$ intersects $X$ in at most $1\cdot 3=3$ points. So \[ \Lambda \subset X \iff |\Lambda \cap X | \geq 4 \text{ points. }\] Asking that $\Lambda$ intersect $X$ imposes one linear condition on the coefficients of $\Lambda \in \mathbb{G}r(1, \PP^3)$, with $\dim \mathbb{G}r(1,\PP^3)=4$. So $F_1(X)$ has dimension $0$ and codimension $4$ in $\mathbb{G}r(1,\PP^3)$. \\

 \begin{lemma}
     For $V$ a $4$-dimensional vector space, with $\mathcal{S} \subset V \otimes \mathcal{O}_\mathbb{G}$ the tautological rank $2$ subbundle on $\mathbb{G}=\mathbb{G}r(1,\PP(V))$ of lines in $\PP V \cong \PP^3$. A form $f$ of degree $d=3$ on $\PP V$ gives rise to a global section $\sigma_f$ of $Sym^d(\mathcal{S}^*)$, whose zero locus is $F_1(X)$, where $X = V(f)$. Thus, when $F_1(X)$ has the expected codimension $\binom{4}{1} = \text{rank}(\text{Sym}^d(\mathcal{S}^*))$ in $\mathbb{G}$, we have \[ [ F_1(X)] = c_{\binom{4}{1}}\left(\text{Sym}^d(\mathcal{S}^*)\right) \] in $A(\mathbb{G}r(1,\PP^1))$. 
 \end{lemma}

That is, the top chern class $c_4(\text{Sym}^3(\mathcal{S}^*))$ in the Chow ring $A(\mathbb{G}r(1,\PP^3))$ counts the number of points in $F_1(X)$, which is the number of lines in $X$. Here, the trivial rank $2$ bundle $\mathcal{S}$ does not have global sections, but $\mathcal{S}^*$ does, and the rational section $\sigma_f \in \Gamma(\text{Sym}^3(\mathcal{S}^*))$ defined by $f$ has the property that a root of $\sigma_f$ corresponds to a line $\Lambda$ in $X$. Since $\mathcal{S}$ has rank $2$, so does $\mathcal{S}^*$, and for $\{a,b\}$ a (fiber-wise) basis of $\mathcal{S}^*$ we have a fiber-wise basis of $\text{Sym}^3(\mathcal{S}^*)$ given by 

 \[ \{ a \otimes a \otimes a, a \otimes a \otimes b, a \otimes b \otimes b, b \otimes b \otimes b \} \] 
 so that $\text{Sym}^3(\mathcal{S}^*)$ has rank $4$. One can show that the total chern class of $\mathcal{S}^*$ is 
 \[ c(\mathcal{S}^*) = 1 + \sigma_{1,0} + \sigma_{1,1} \in A(\mathbb{G}r(1,\PP^3))\] 
 in the notation for Schubert classes given in \cite{eisenbud20163264}. Following Eisenbud, here we leave as an 

 \begin{exercise} Using the splitting principle for a rank $2$ vector bundle $\mathcal{E}$ (i.e., that $c(\mathcal{E}) = (1+\alpha)(1+\beta)$ for $\alpha$ and $\beta$ each the first chern class of a line bundle), we have that 

\begin{align*}
c_{top}\text{Sym}^3(\mathcal{E}) = c_4(\text{Sym}^3(\mathcal{E})) &= 18 c_1^2(\mathcal{E})c_2(\mathcal{E}) + 9 c_2^2(\mathcal{E}) 
 \end{align*}
 \end{exercise}
 
Continuing, $\mathcal{E}=\mathcal{S}^*$ with $c_1(\mathcal{S}^*) = \sigma_{1,0}$ and $c_2(\mathcal{S}^*) = \sigma_{1,1}$. This gives, via relations in $A(\mathbb{G}r(1, \PP^3)$ as given in \cite{eisenbud20163264}\\

\begin{align*}
    18 \sigma_{1,0}^2 \sigma_{1,1} + 9 \sigma_{1,1}^2 &= 18(\sigma_{1,1} + \sigma_{2,0} ) \cdot \sigma_{1,1} + 9 \sigma_{2,2}\\
    &= 18(\sigma_{2,2} + 0 ) + 9\sigma_{2,2}\\
    &= 27 \sigma_{2,2}
\end{align*} from $\sigma_{1,1}^2 = \sigma_{2,2}, \sigma_{1,0}^2 = \sigma_{1,1} + \sigma_{2,0},$ and $\sigma_{2,0}\sigma_{1,1}=0$. This gives \[ \deg(c_4(\text{Sym}^3(\mathcal{S}^*))) = 27. \] 

\end{proof}

While the number of lines on a cubic surface is finite, the number of conics on a cubic surface is infinite. To have a space $X$ where the number of degree $d$ rational curves in $X$ is finite, it is helpful to first consider a quintic threefold $Q=V(f)$ for $f\in \Gamma(\PP^4, \mathcal{O}_{\PP^4}(5))$. $Q$ is a compact Calabi-Yau variety since $Q$ is compact as a projective hypersurface in $\PP^4$, the canonical bundle $K_Q \cong \mathcal{O}_Q$, and $c_1(TQ)=0$. Here, the canonical bundle $K_Q \cong \mathcal{O}_Q(5 - 4 - 1)=\mathcal{O}_Q$ as a smooth degree $5$ hypersurface of $\PP^4$ \cite{hartshorne} Exc. 8.4(e).\\

\begin{proposition} For $Q$ a smooth quintic threefold in $\PP^4$, we have $c_1(TQ)=0$. 
    
\end{proposition}

\begin{proof} 
To show that $c_1(TQ)=0$, we note that the short exact sequence

\[ 0 \rightarrow TQ \rightarrow T\PP^4\bigg|_Q \rightarrow N_Q \rightarrow 0 \] 

becomes \[ 0 \rightarrow TQ \rightarrow T\PP^4\bigg|_Q \rightarrow \mathcal{O}_{\PP^4}(5) \rightarrow 0 \] using adjunction, since $Q$ is smooth. This implies

\begin{align*}
    c(TQ) &= \frac{(1+h\bigg|_Q)^5}{1+5h\bigg|_Q}\\
    &= \frac{1 + 5h\bigg|_Q + 10 (h\bigg|_Q)^2 + 10 (h\bigg|_Q)^3 }{1 + 5h\bigg|_Q}\\
    &= 1 + 10(h\bigg|_Q)^2 - 40 (h\bigg|_Q)^3
\end{align*}

where $h$ is the hyperplane class of $\PP^4$, so that $c_1(TQ)=0$ as claimed. \end{proof}


A similar chern class argument as above for a nonsingular cubic surface shows that a nonsingular quintic 3-fold contains 2875 lines. Here, we give this well-known proof of this fact, relying on the structure of $A( \mathbb{G}r(1, \mathbb{P}^3))$ as described in \cite{eisenbud20163264}.

\begin{proof} For $Q = V(F)$ with $F \in \Gamma( \PP^4, \mathcal{O}_{\PP^4}(5))$, we first note that $Gr(2,\A^5) \cong \mathbb{G}r(1, \PP^4)$ via projectivizing, and that $\mathbb{G}r(1, \PP^4)$ has a tautological rank $2$ vector bundle $\mathcal{S}$ as for the cubic. Viewing $T \text{Gr}(2, \A^5) \cong \text{Hom}_{\text{v. bdl}}(\mathcal{S}, \mathcal{S}^\perp)$ similarly shows $\dim \text{Gr}(2,\A^5) = 2 \cdot 3 = 6$. As $Q = V(f)$ from $f\in \Gamma(\mathcal{O}_{\PP^4}(5)),$ we have a corresponding global section $\sigma_f \in \Gamma( \text{Sym}^5(\mathcal{S}^*)$, with $\text{rank}(\text{Sym}^5(\mathcal{S}^*))=6$. Since the line $\Lambda \subset \PP^4$ satisfies \[ \Lambda \subset Q \iff |\Lambda \cap Q| \geq 6 \text{ points } \] from B\'{e}zout's theorem, we see that $\Lambda$ (viewed as a point of $\mathbb{G}r(1,\PP^4)$ being contained in $Q$ imposes $6$ linear conditions on $\text{Gr}(1, \PP^4)$, which is again 6-dimensional over $\C$. Letting $F_1(Q)$ denote the variety of lines contained in $Q$, we have that \[ [F_1(Q)] = c_6(\text{Sym}^5(\mathcal{S}^*)) \] in the Chow ring $A^\bullet( \mathbb{G}(1, \PP^4))$. Direct computation in Macaulay2 yields that the top chern class of $\text{Sym}^5(\mathcal{S}^*) = 2875$, using the ring structure of $A^\bullet( \mathbb{G}(1, \PP^4))$:\\

\FloatBarrier
\begin{figure}[h]
\includegraphics[width=\textwidth, trim = {1cm 5.2cm 12cm 4cm}, clip]{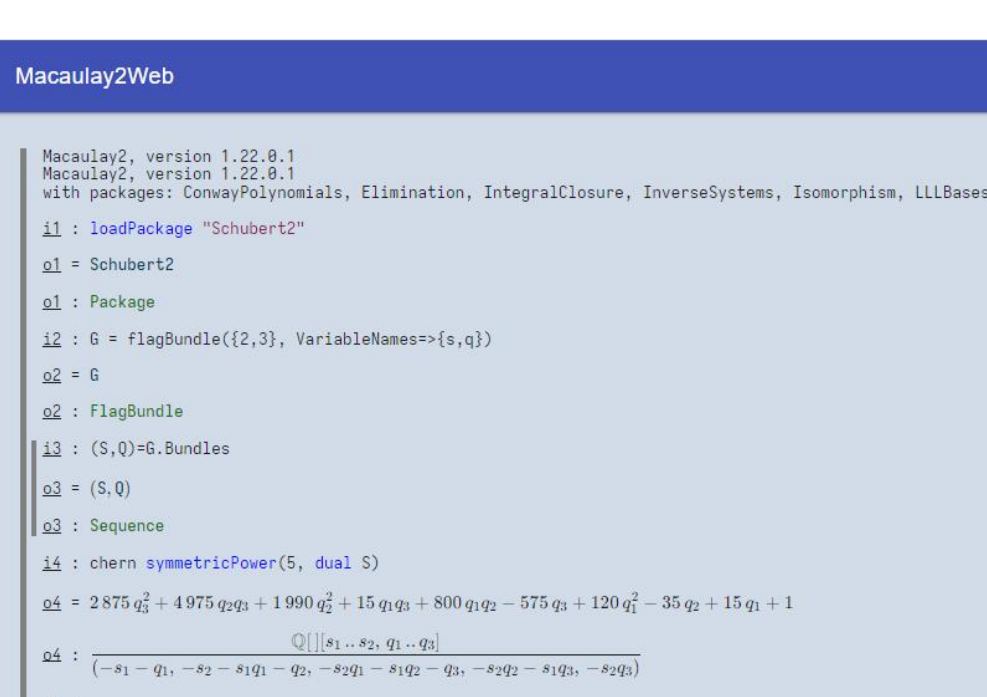}
\caption{Macaulay2 output for computing 2875 lines on the quintic threefold}
\end{figure} 
\FloatBarrier

\end{proof}

\section{Differential Geometry Overview and Equivariant cohomology}

\subsection{Differential Geometry Overview}
Often in topology, if we are given a diffeomorphic action of a compact Lie group $G$ on a manifold $M$, then we can understand the topology of $M$ better by considering the action of $G$ on $M$. Consider $M=S^2 \hookrightarrow \R^3$ and let $G=S^1$ be the compact Lie group acting on $S^2$ by rotation around the $z-$axis:

\FloatBarrier
\begin{figure}[h]
\centering

\definecolor{figureRed}{RGB}{255,55,40}
\definecolor{figureBlue}{RGB}{0,83,221}

\resizebox{.8\textwidth}{!}{%
\begin{tikzpicture}[
  line cap=round,
  line join=round,
  >=latex,
  every node/.style={font=\small}
]

\node[figureRed,align=center,font=\large] at (-2.55,0.15)
  {$2\;\mathbb{T}$-fixed points\\[-1pt]
   $\implies\;\chi(S^2)=2$};

\begin{scope}[xshift=0.5in]

\draw[black,line width=0.75pt] (0.8,0) circle[radius=2.15];

\draw[black,dashed,line width=0.65pt] (0.8,-2.75) -- (0.8,2.75);

\draw[figureBlue,line width=0.7pt,dashed]
  (-0.89,1.33) arc[start angle=180,end angle=360,x radius=1.69,y radius=0.22];
\draw[figureBlue,line width=0.7pt]
  (-0.89,1.33) arc[start angle=180,end angle=0,x radius=1.69,y radius=0.22];

\fill[figureBlue] (0.8,2.15) circle[radius=3.1pt];
\fill[figureBlue] (0.8,-2.15) circle[radius=3.1pt];

\draw[figureRed,->,line width=0.7pt]
  (0.30,2.55)
  arc[start angle=180,end angle=505,x radius=0.50,y radius=0.18];

\draw[figureRed,->,line width=0.7pt]
  (0.30,-2.55)
  arc[start angle=180,end angle=505,x radius=0.50,y radius=0.18];

\draw[figureRed,->,line width=0.7pt]
  (-3.75,1.15)
  .. controls (-3.10,1.95) and (-2.20,2.70) .. (-0.10,2.55)
  .. controls (-0.10,2.40) and (-0.02,2.38) .. (0.20,2.32);

\draw[figureRed,->,line width=0.7pt]
  (-3.70,-0.90)
  .. controls (-3.05,-1.65) and (-2.20,-2.75) .. (-0.05,-2.55)
  .. controls (-0.10,-2.40) and (-0.02,-2.38) .. (0.20,-2.32);

\draw[figureBlue,->,line width=0.7pt]
  (4.00,1.56) -- (2.58,1.35);

\node[figureBlue,anchor=west,align=left,font=\small] at (4.02,1.58)
  {$\mathbb{T}$-orbit\\[-1pt] of a non-fixed point};

\end{scope}

\end{tikzpicture}%
}

\caption{A localization in equivariant cohomology argument implies that $\chi(S^2)=2$.}
\end{figure}
\FloatBarrier

If we embed $S^2 \hookrightarrow \R^3$ via the inclusion \[ S^2 = \{(x,y,z) \text{ }|\text{ } x^2+y^2+z^2=1\}, \] this action corresponds to the Morse function $h: S^2 \rightarrow \R$ given by the height function \[ h(x,y,z) = z. \] The function $h$ is Morse since it has isolated, non-degenerate critical points $\{(0,0,\pm 1)\}$ \cite{audin2013morse}. From the Morse function $h$, we can form the generic section $s \in \Gamma(S^2, TS^2)$ given by \[ s(p) = \frac{ -\nabla h(p) }{|| \nabla h(p) ||} \] 

Here, the \textbf{Poincar\'{e}-Hopf} theorem yields the Euler characteristic of $S^2$:\\

\[ \chi(S^2) = \sum_{s(p)=0} \text{ind}(p) \] where the Poincar\'{e}-Hopf index of $p$ is given by the degree of the map from $\partial B_n(p) \rightarrow S^{n-1}$ for $B_n(p)$ a neighborhood of $p$ in $M$. Here, both the north and south poles $(0,0,\pm 1)$ have Poincar\'{e}-Hopf index 1, so we see that $\chi(S^2)=2$. Alternatively, we can use the Gauss-Bonnet theorem

\[ 2\pi \chi(M) = \int_M K dV \] for $K$ the Gaussian curvature, and $dV$ the volume form. For $M=S^2$, we have that $K(p)$, the Gaussian curvature at the point $p$, is the product of the max and min sectional curvatures, which are each $\frac{1}{r}$. This gives 

\begin{align*}
    2\pi \chi(S^2) &= \int_{S^2} \frac{1}{r^2} r^2 \sin\phi d\theta \wedge d\phi \\
    &= \int_{\theta=0}^{2\pi} \int_{\phi=0}^\pi \sin\phi d\theta d\phi = 2\pi(2) 
\end{align*} so that $\chi(S^2)=2$.\\

One method of calculating the euler characteristic of an orientable smooth real manifold $M$ which is useful for generalizing to methods involving chern classes in the holomorphic setting, is to integrate the euler class of the tangent bundle: \[ \chi(M) = \int_{M} e[TM] \] 

Here we work out an example for $M=S^2$. 

\begin{proof}

Consider the two-sphere with radius $\rho$, denoted $S^2(\rho)$, with metric 

\begin{align*} g_{ij} = \rho^2 \left(\begin{matrix} \sin^2\phi & 0 \\ 0 & 1 \end{matrix} \right)\end{align*} and set $\rho=1$. Here, $\theta$ is azimuthal and $\phi$ is angle of declination from the north pole, $N$. Then the Christoffel symbols are 

$$\Gamma_{\theta\theta}^\theta = 0, \Gamma_{\theta\phi}^\theta=\Gamma_{\phi\theta}^\theta=\frac{\cos\phi}{\sin\phi}, \Gamma_{\phi\phi}^\theta=0,$$

$$\Gamma_{\theta\theta}^\phi = \sin\phi\cos\phi, \Gamma_{\theta\phi}^\phi=\Gamma_{\phi\theta}^\phi=0, \Gamma_{\phi\phi}^\phi=0.$$ Let $ \mathcal{B} = \{\partial_\theta, \partial_\phi\}$ denote an ordered basis for the tangent space of each point of $S^2(\rho)$ away from the poles. From the convention $$a_X e_j = \sum \omega_i^j(X) e_i$$ and $F = dA + A \wedge A$ for $A$ the connection matrix from arranging the $\omega_i^j$ into a matrix $A$, we have 


$$A = \left( \begin{matrix} \cot\phi d\phi & \cot\phi d\theta \\ -\sin\phi\cos\phi d\theta & 0 \end{matrix} \right) $$ so that 

$$dA = \left(\begin{matrix} 0 & -\csc^2(\phi) d\phi \wedge d\theta \\ -\cos^2(\phi) + \sin^2(\phi) d\phi \wedge d\theta & 0 \end{matrix}\right) $$

and

$$ A \wedge A = \left( \begin{matrix} \cot\phi d\phi & \cot\phi d\theta \\ -\sin\phi\cos\phi d\theta & 0 \end{matrix} \right) \wedge \left( \begin{matrix} \cot\phi d\phi & \cot\phi d\theta \\ -\sin\phi\cos\phi d\theta & 0 \end{matrix} \right) $$

gives $$ F = dA + A \wedge A = \left( \begin{matrix} 0 & 1 d\theta \wedge d\phi \\ -\sin^2\phi d\theta \wedge d\phi & 0 \end{matrix}\right). $$ 

In the orthonormal basis $\tilde{B} = \{ \frac{\partial_\theta}{\sin\phi}, \partial_\phi \}$ we have that $\tilde{e}_1 = \frac{1}{\sin\phi}e_1, \tilde{e_2} = 1\cdot e_2$ gives the change-of-basis matrix $$ a = \left[ \begin{matrix} \frac{1}{\sin\phi} & 0 \\ 0 & 1 \end{matrix}\right]$$ to say $\tilde{e} = ea$ as row vectors. The curvature matrix transforms as $$ \tilde{\Omega} = a^{-1}\Omega a $$ which gives 

\begin{align*} \tilde{\Omega} &= \left( \begin{matrix} \sin\phi & 0 \\ 0 & 1 \end{matrix}\right) \left(\begin{matrix} 0 & 1 \\ -\sin^2\phi & 0 \end{matrix}\right) \left(\begin{matrix} \frac{1}{\sin\phi} & 0 \\ 0 & 1 \end{matrix}\right)  \\
&= \left( \begin{matrix} 0 & \sin\phi\\ -\sin\phi & 0 \end{matrix}\right) \otimes d\theta \wedge d\phi \end{align*}

whose Pfaffian is 

\begin{align*}
Pf(\tilde{\Omega}) &= \frac{1}{2} ( \tilde{\Omega}_{1,2} - \tilde{\Omega}_{2,1} )\\
&= \frac{1}{2} ( \sin\phi - (-\sin\phi) ) = \sin\phi \end{align*} tensored with $d\theta \wedge d\phi$ gives

\begin{align*}
\chi(M) = \int_{M} e[TM]= \int_{S^2} Pf(\frac{1}{2\pi} \tilde{\Omega}) &= \int_{S^2} e[TS^2] \\
&= \int_{S^2} \frac{\sin\phi}{2\pi}d\theta \wedge d\phi\\
&= \frac{1}{2\pi} 2(2\pi) = 2 = \chi(S^2). \end{align*}

\end{proof}

We can compute $\chi(S^2)$ by using Morse homology as follows. Given a Morse function $h: S^2 \rightarrow \R$ (with isolated critical points), the Morse lemma implies that locally, $h$ can be written as \[ h(p) + (-x_1^2 - x_2^2 - \cdots -x_\iota^2) + (x_{\iota+1}^2 + \cdots + x_n^2) \] in a coordinate chart $U$ around $p$. 

\begin{definition} The Morse index of $p$ is $\iota$, the number of coefficients above with a minus sign. 
\end{definition} 

Now we build the Morse complex $C_h$ with points in index $i$ placed in homological degree $i$, and $\partial: (C_h)_i \rightarrow (C_h)_{i-1}$ given by (negative) gradient trajectory flow lines

\FloatBarrier
\begin{figure}[h]
\centering
\begin{tikzpicture}[
    scale=1.3,
    >={Stealth[length=2.5mm, width=1.5mm]},
    flowline/.style={
        decoration={markings, mark=at position 0.35 with {\arrow{>}}, mark=at position 0.75 with {\arrow{>}}},
        postaction={decorate},
        blue!80!black,
        thick
    },
    backline/.style={
        decoration={markings, mark=at position 0.35 with {\arrow{>}}, mark=at position 0.75 with {\arrow{>}}},
        postaction={decorate},
        dashed,
        gray!70,
        thin
    }
]

    \draw[thick] (0,0) circle (2cm);
    \draw[thick] (-2,0) arc (180:360:2 and 0.5); 
    \draw[dashed, gray!80] (2,0) arc (0:180:2 and 0.5); 

    
    \draw[backline] (0,2) arc (90:270:0.4 and 2);
    \draw[backline] (0,2) arc (90:270:1.25 and 2);
    \draw[backline] (0,2) arc (90:-90:0.4 and 2);
    \draw[backline] (0,2) arc (90:-90:1.25 and 2);

    \draw[flowline] (0,2) -- (0,-2);
    \draw[flowline] (0,2) arc (90:270:0.85 and 2);
    \draw[flowline] (0,2) arc (90:270:1.7 and 2);
    \draw[flowline] (0,2) arc (90:-90:0.85 and 2);
    \draw[flowline] (0,2) arc (90:-90:1.7 and 2);

    \filldraw[blue] (0,2) circle (2pt) node[above, text=black] {$N=p$};
    \filldraw[blue] (0,-2) circle (2pt) node[below, text=black] {$S=q$};

    \begin{scope}[shift={(5, 1.8)}]
        \draw[thick] (-1, -1) rectangle (1, 1);
        \draw[->, thick] (0,0) -- (1.4, 0);
        \draw[->, thick] (0,0) -- (-1.4, 0);
        \draw[->, thick] (0,0) -- (0, 1.4);
        \draw[->, thick] (0,0) -- (0, -1.4);
        \node[right] at (1.6, 0.4) {\Large $T_p S^2$};
        \node[right] at (1.6, -0.4) {index 2};
    \end{scope}

    \begin{scope}[shift={(5, -1.8)}]
        \draw[thick] (-1, -1) rectangle (1, 1);
        \draw[<-, thick] (0,0) -- (1.4, 0);
        \draw[<-, thick] (0,0) -- (-1.4, 0);
        \draw[<-, thick] (0,0) -- (0, 1.4);
        \draw[<-, thick] (0,0) -- (0, -1.4);
        \node[right] at (1.6, 0.4) {\Large $T_q S^2$};
        \node[right] at (1.6, -0.4) {index 0};
    \end{scope}

\end{tikzpicture}
\caption{Building the Morse complex on $S^2$}
\label{fig:torusfibersS2}
\end{figure}
\FloatBarrier

This gives the Morse complex\[ C_h: 0 \rightarrow \Z\left<p\right> \stackrel{\cdot 0}{\rightarrow} 0 \stackrel{\cdot 0}{\rightarrow} \Z\left< q \right> \rightarrow 0 \] with $\Z\left<p\right>$ in homological degree $2$, and $\Z\left<q\right>$ in homological degree $0$. This gives 

\begin{align*}
    H_*(C_h) = \begin{cases} \Z & i=0,2\\ 0 &\text{ else} \end{cases}
\end{align*}

and

\begin{align*}
    \chi(S^2) &= \sum_{i=0}^2 (-1)^i b_i\\
    &= 2.
\end{align*}

To see that homology is a homotopy invariant, we can also compute the Morse complex of the deformed sphere. 
\begin{figure}[h]
\centering
\begin{tikzpicture}[
    scale=0.95,
    line cap=round,
    line join=round,
    >={Stealth[length=1.8mm,width=1.2mm]},
    pqflow/.style={
        blue!70!black,
        thick,
        decoration={markings, mark=at position 0.58 with {\arrow{>}}},
        postaction={decorate}
    },
    pqback/.style={
        blue!40,
        thin,
        dashed,
        decoration={markings, mark=at position 0.58 with {\arrow{>}}},
        postaction={decorate}
    },
    pwflow/.style={
        red!75!orange,
        thick,
        decoration={markings, mark=at position 0.58 with {\arrow{>}}},
        postaction={decorate}
    },
    pwback/.style={
        red!65!orange,
        thick,
        dashed,
        decoration={markings, mark=at position 0.58 with {\arrow{>}}},
        postaction={decorate}
    },
    wqflow/.style={
        purple!85!magenta,
        thick,
        decoration={markings, mark=at position 0.58 with {\arrow{>}}},
        postaction={decorate}
    },
    wqback/.style={
        purple!55!magenta,
        thick,
        dashed,
        decoration={markings, mark=at position 0.58 with {\arrow{>}}},
        postaction={decorate}
    },
    contour/.style={gray!70,thin},
    tbox/.style={black,thick},
    every node/.style={font=\small}
]

\begin{scope}[shift={(-3.0,0)}]

    \def\blobpath{(0,-3.30)
        .. controls (-2.60,-3.10) and (-2.55,-1.00) .. (-2.10,0.70)
        .. controls (-1.78,2.10) and (-1.18,2.95) .. (-0.68,2.85)
        .. controls (-0.12,2.75) and (-0.42,1.10) .. (-0.52,0.20)
        .. controls (-0.58,-0.25) and (-0.26,-0.62) .. (0,-0.55)
        .. controls (0.26,-0.62) and (0.58,-0.25) .. (0.52,0.20)
        .. controls (0.42,1.10) and (0.12,2.75) .. (0.68,2.85)
        .. controls (1.18,2.95) and (1.78,2.10) .. (2.10,0.70)
        .. controls (2.55,-1.00) and (2.60,-3.10) .. (0,-3.30)}

    \shade[ball color=cyan!5,opacity=0.28] \blobpath;
    \draw[blue!70!black,very thick] \blobpath;

    \draw[contour] (-2.02,0.96)
        arc[start angle=180,end angle=360,x radius=0.82,y radius=0.19];
    \draw[contour,dashed] (-2.02,0.96)
        arc[start angle=180,end angle=0,x radius=0.82,y radius=0.19];

    \draw[contour] (0.38,0.96)
        arc[start angle=180,end angle=360,x radius=0.82,y radius=0.19];
    \draw[contour,dashed] (0.38,0.96)
        arc[start angle=180,end angle=0,x radius=0.82,y radius=0.19];

    \draw[contour] (-2.05,-2.20)
        arc[start angle=180,end angle=360,x radius=2.05,y radius=0.44];
    \draw[contour,dashed] (-2.05,-2.20)
        arc[start angle=180,end angle=0,x radius=2.05,y radius=0.44];

    \filldraw[black] (-0.68,2.85) circle (2pt) node[above] {$p_1$};
    \filldraw[black] ( 0.68,2.85) circle (2pt) node[above] {$p_2$};
    \filldraw[black] (0,-0.55) circle (2pt);
    \node at (0.36,-0.10) {$w$};
    \filldraw[black] (0,-3.30) circle (2pt) node[below] {$q$};

    \begin{scope}
        \clip \blobpath;

        \draw[pqflow] (-0.68,2.85)
            .. controls (-1.95,1.78) and (-2.08,-1.55) .. (0,-3.30);
        \draw[pqflow] (0.68,2.85)
            .. controls (1.95,1.78) and (2.08,-1.55) .. (0,-3.30);

        \draw[pqback] (-0.68,2.85)
            .. controls (-1.28,1.88) and (-1.18,-1.82) .. (0,-3.30);
        \draw[pqback] (0.68,2.85)
            .. controls (1.28,1.88) and (1.18,-1.82) .. (0,-3.30);

        \draw[pwflow] (-0.68,2.85)
            .. controls (-0.66,2.00) and (-0.34,0.48) .. (0,-0.55);
        \draw[pwflow] (0.68,2.85)
            .. controls (0.66,2.00) and (0.34,0.48) .. (0,-0.55);

        \draw[pwback] (-0.68,2.85)
            .. controls (-0.88,1.72) and (-0.56,0.16) .. (0,-0.55);
        \draw[pwback] (0.68,2.85)
            .. controls (0.88,1.72) and (0.56,0.16) .. (0,-0.55);

        \draw[wqflow] (0,-0.55)
            .. controls (-0.24,-1.34) and (-0.16,-2.42) .. (0,-3.30);

        \draw[wqback] (0,-0.55)
            .. controls (0.30,-1.30) and (0.22,-2.46) .. (0,-3.30);
    \end{scope}

\end{scope}

\begin{scope}[shift={(3.6,2.10)}]
    \draw[tbox] (-0.72,-0.72) rectangle (0.72,0.72);
    \draw[{Stealth}-{Stealth}, blue!80, thick] (-0.95,0) -- (0.95,0);
    \draw[{Stealth}-{Stealth}, blue!80, thick] (0,-0.95) -- (0,0.95);
    \node[right=1.15cm,align=left] at (0,0) {$T_{p_1}X$\\ index 2};
\end{scope}

\begin{scope}[shift={(7.8,2.10)}]
    \draw[tbox] (-0.72,-0.72) rectangle (0.72,0.72);
    \draw[{Stealth}-{Stealth}, blue!80, thick] (-0.95,0) -- (0.95,0);
    \draw[{Stealth}-{Stealth}, blue!80, thick] (0,-0.95) -- (0,0.95);
    \node[right=1.15cm,align=left] at (0,0) {$T_{p_2}X$\\ index 2};
\end{scope}

\begin{scope}[shift={(5.7,-0.95)}]
    \draw[tbox] (-0.72,-0.72) rectangle (0.72,0.72);
    \draw[-{Stealth}, red!80, thick] (-0.95,0) -- (0,0);
    \draw[-{Stealth}, red!80, thick] (0.95,0) -- (0,0);
    \draw[{Stealth}-{Stealth}, blue!80, thick] (0,-0.95) -- (0,0.95);
    \node at (0,0) {$\times$};
    \node[right=1.15cm,align=left] at (0,0) {$T_wX$\\ index 1};
\end{scope}

\begin{scope}[shift={(5.7,-4.00)}]
    \draw[tbox] (-0.72,-0.72) rectangle (0.72,0.72);
    \draw[-{Stealth}, red!80, thick] (-0.95,0) -- (0,0);
    \draw[-{Stealth}, red!80, thick] (0.95,0) -- (0,0);
    \draw[-{Stealth}, red!80, thick] (0,0.95) -- (0,0);
    \draw[-{Stealth}, red!80, thick] (0,-0.95) -- (0,0);
    \node at (0,0) {$\times$};
    \node[right=1.15cm,align=left] at (0,0) {$T_qX$\\ index 0};
\end{scope}

\end{tikzpicture}
\caption{Building the Morse complex for a homotopy deformation of $S^2$}
\end{figure}

This gives the Morse complex

\[ 0 \rightarrow \makecell{\Z\left<p_1\right> \\\newline \oplus \\\newline \Z\left<p_2\right> } \stackrel{\left( \begin{matrix} 1&  1 \end{matrix}\right)}{\rightarrow} \Z\left<w\right> \stackrel{\cdot 0}{\rightarrow } \Z\left<q\right> \rightarrow 0 \] 

More generally, we have the following exercise. 

\begin{exercise}
    Compute $\chi(\Sigma_g)$ using the Morse complex $C_f$ from the height function for the Morse-Smale, tilted genus $g$ Riemann surface $\Sigma_g$.  
\end{exercise}

Furthermore, $S^2$ is diffeomorphic to $\PP^1_\C$ as a smooth manifold, and the function $h$ corresponds to the moment map of toric topology \cite{C-L-S}. Here we see torus fibers of $h$ as given by Figure~\ref{fig:torusfibersS2}. 

Here, the image $h(\PP^1) = [-1,1]$ can be viewed as the toric polytope $P$, whose inner normal fan gives the fan for $\PP^1(\C)$ \cite{C-L-S}.

Up to re-scaling $P$, we also have that the symplectic moment map $\phi: \PP^1 \rightarrow \R$ can be written as \[ \phi: \PP^1\rightarrow \R, \hspace{2cm} \phi([z_0:z_1]) = \frac{ |z_1|^2}{|z_0|^2+|z_1|^2}. \]

with image $[0,1] \subset \R$. 
Topologically, $S^2$ has a cell decomposition \[ S^2 = \{pt \} \sqcup e^2_\R \] where $e^2_\R$ is an open $2-$ball, whose boundary is identified to $\{pt\}$ in the construction of $S^2$. This corresponds to the Schubert cell decomposition of $\PP^1$, which we describe below in Section~\ref{schubcell}.

\FloatBarrier
\begin{figure}[h]
\centering
\resizebox{0.92\textwidth}{!}{%
\begin{tikzpicture}[
    x=1cm,y=1cm,
    line cap=round,
    line join=round,
    font=\small
]

\node at (0,2.55) {$\displaystyle
\PP^1_{\C}=\{[1\!:\!0]\}\sqcup U,\qquad
U=\{[z\!:\!1]\mid z\in\C\}\cong \A^1_{\C}\cong e^2_{\R}$};

\begin{scope}[shift={(-4.3,0.05)}]
    \shade[inner color=white, outer color=purple!18] (0,0) circle (1.05);
    \draw[purple!85!magenta, line width=1.1pt] (0,0) circle (1.05);

    \draw[black, line width=0.9pt]
        (-1.05,0) arc[start angle=180,end angle=360,x radius=1.05,y radius=0.24];
    \draw[purple!55!magenta, dashed, dash pattern=on 5pt off 4pt, line width=0.9pt]
        (-1.05,0) arc[start angle=180,end angle=0,x radius=1.05,y radius=0.34];

    \node at (0,1.38) {$\PP^1_{\C}$};
    \fill[blue!85] (0,-1.05) circle (0.08);
    \node at (0,-1.42) {$[1:0]$};
\end{scope}

\node at (-2.25,0.05) {$\displaystyle =$};
\node at (-0.55,0.05) {$\displaystyle \left\{\;\bullet\;\right\}$};
\fill[blue!85] (-0.55,0.05) circle (0.08);
\node at (-0.55,-1.42) {$[1:0]$};
\node at (1.05,0.05) {$\displaystyle \sqcup$};

\begin{scope}[shift={(4.0,0.05)}]
    \shade[inner color=white, outer color=purple!16] (0,0) circle (1.2);
    \draw[purple!85!magenta, dashed, dash pattern=on 5pt off 4pt, line width=1.1pt]
        (0,0) circle (1.2);
    \node at (0,0) {$U$};
    \node at (0,-1.50) {$\{[z\!:\!1]\mid z\in\C\}$};
\end{scope}

\end{tikzpicture}%
}
\caption{Schubert cell decomposition for $\PP^1$}
\end{figure}
\FloatBarrier

By viewing $S^2$ instead as a complex manifold, we see that \[ \chi(S^2) = \chi(\PP^1_\C) = \int_{\PP^1} c_{\text{top}}(T\PP^1) \]


\subsection{Equivariant Cohomology} Excellent references for this topic include Tu \cite{Tu+2020}, \cite{MSTextbook}, and \cite{coxkatxMSAG}.

\begin{defn} Recall that a group $G$ acts freely on $M$ if the stabilizer of every point $x\in M$ \[ \text{Stab}(x) = \{g\in G\text{ }|\text{ }gx=x\} \] equals $\{1\}$ for all $x\in M$. 
\end{defn} 

Let $G$ be a compact connected Lie group and $M$ a $C^\infty$ $G-$manifold. If $G$ does not act on $M$ freely (i.e., in the presence of stabilizers), then the quotient $\faktor{M}{G}$ can be poorly behaved. To correct this, we let $E$ be a contractible space on which $G$ acts freely. Then

\begin{defn} The \textbf{$G$-equivariant cohomology of $M$} is the ``usual" singular cohomology of the homotopy quotient \[ M_G = \faktor{(E\times M)}{G}.\]  \end{defn} 

\begin{remark} To get $EG$, one considers a principal $G$-bundle $EG \rightarrow BG$ with weakly contractible total space such that $G$ acts freely on $EG$. Now the classifying space of $G$ is given by $BG = \faktor{EG}{G}$. \end{remark}

Localization in equivariant cohomology can be described as follows. For $\phi$ a G-equivariant cohomology class, let $F$ run over the $G-$fixed loci of $M$. Let $M$ be a compact holomorphic $G$-manifold (this implies orientable, so that we have an Euler class and Poincar\`{e} Duality to push forward cohomology classes). Then we have: 

\begin{defn} The \textbf{localization formula}

\label{eqn: localiz}
\[ \int_M \phi = \sum_F \int_F \frac{ i^*\phi}{e(N_{F/M})}. \] \end{defn} 

\begin{remark} This is the \textbf{integral version} of localization. \end{remark} 

The key example that we'll use comes from $G=\T$ a finite product of copies of $S^1$ and $\C^*$ acting holomorphically on $M$. If $M$ has a finite number $n$ of $\T$-fixed points, then $\chi(M)=n$. Note that $\chi(M) = e(T_M)$. There is a natural $\T$-action on $T_M$, inducing a bundle on $M_\T$, the equivariant tangent bundle. By the localization formula Equation~\ref{eqn: localiz}, \[ \int_M e(T_M) = \sum_F \int_F \left( \frac{ e(N_{F/M})}{e(N_{F/M})} \right) = \sum_{j=1}^n 1 = n. \]

\begin{remark} When the CW complex $X$ has the structure of finite cellular complex, then \[ \chi(X) = \sum_{i=0}^{\text{dim} X} (-1)^i \cdot \{ \# i-\text{dimensional cells} \}\] (using the real dimension here). \end{remark}

\begin{ex} A finite graph $G=(V,E)$ has Euler characteristic \[ \chi(G) = (\# \text{ vertices }) - ( \# \text{ edges }) = \sum_{i=0}^1 (-1)^i\cdot \left( \# i-\text{cells of }G \right). \] \end{ex}

\begin{ex} $S^1 \times S^1$. Let $\T = S^1$ act on $M = \T^2 = S^1 \times S^1$ by rotations 


\FloatBarrier
\begin{figure}[h]
\centering
\resizebox{0.72\textwidth}{!}{%
\begin{tikzpicture}[
  line cap=round,
  line join=round,
  gridfront/.style={
    draw=gray!48,
    line width=.38pt
  },
  gridback/.style={
    draw=gray!30,
    line width=.34pt,
    dash pattern=on 2.3pt off 2.1pt
  },
  orbitfront/.style={
    draw=blue!78!black,
    line width=1.9pt
  },
  orbitback/.style={
    draw=blue!45,
    line width=1.55pt,
    dash pattern=on 5.5pt off 4.2pt
  }
]

\def\Rmaj{3.10}
\def\rmin{1.20}
\def\xs{1.10}
\def\ys{0.85}
\def\q{0.573576} 
\def\s{0.819152} 

\shade[
  top color=gray!1,
  bottom color=gray!22
]
  plot[
    domain=0:360,
    samples=220,
    variable=\u
  ]
  (
    {\xs*(\Rmaj
      +\rmin*\q/sqrt(\q*\q+\s*\s*sin(\u)*sin(\u)))*cos(\u)},
    {\ys*(\q*\Rmaj*sin(\u)
      +\rmin*sin(\u)/sqrt(\q*\q+\s*\s*sin(\u)*sin(\u)))}
  )
  -- cycle;

\foreach \u in {20,50,80,110,140,160}{
  \draw[gridback]
    plot[
      domain=0:360,
      samples=90,
      variable=\v
    ]
    (
      {\xs*(\Rmaj+\rmin*cos(\v))*cos(\u)},
      {\ys*(\q*(\Rmaj+\rmin*cos(\v))*sin(\u)
        +\s*\rmin*sin(\v))}
    );
}

\foreach \u in {200,230,260,290,320,340}{
  \draw[gridfront]
    plot[
      domain=0:360,
      samples=90,
      variable=\v
    ]
    (
      {\xs*(\Rmaj+\rmin*cos(\v))*cos(\u)},
      {\ys*(\q*(\Rmaj+\rmin*cos(\v))*sin(\u)
        +\s*\rmin*sin(\v))}
    );
}

\foreach \v in {-120,-75,-35,0,35,75,120,180}{
  \draw[gridback]
    plot[
      domain=0:180,
      samples=90,
      variable=\u
    ]
    (
      {\xs*(\Rmaj+\rmin*cos(\v))*cos(\u)},
      {\ys*(\q*(\Rmaj+\rmin*cos(\v))*sin(\u)
        +\s*\rmin*sin(\v))}
    );

  \draw[gridfront]
    plot[
      domain=180:360,
      samples=90,
      variable=\u
    ]
    (
      {\xs*(\Rmaj+\rmin*cos(\v))*cos(\u)},
      {\ys*(\q*(\Rmaj+\rmin*cos(\v))*sin(\u)
        +\s*\rmin*sin(\v))}
    );
}

\fill[white]
  plot[
    domain=0:360,
    samples=220,
    variable=\u
  ]
  (
    {\xs*(\Rmaj
      -\rmin*\q/sqrt(\q*\q+\s*\s*sin(\u)*sin(\u)))*cos(\u)},
    {\ys*(\q*\Rmaj*sin(\u)
      -\rmin*sin(\u)/sqrt(\q*\q+\s*\s*sin(\u)*sin(\u)))}
  )
  -- cycle;

\draw[
  gray!72,
  line width=.82pt
]
  plot[
    domain=0:360,
    samples=220,
    variable=\u
  ]
  (
    {\xs*(\Rmaj
      +\rmin*\q/sqrt(\q*\q+\s*\s*sin(\u)*sin(\u)))*cos(\u)},
    {\ys*(\q*\Rmaj*sin(\u)
      +\rmin*sin(\u)/sqrt(\q*\q+\s*\s*sin(\u)*sin(\u)))}
  )
  -- cycle;

\draw[
  gray!62,
  line width=.68pt
]
  plot[
    domain=0:360,
    samples=220,
    variable=\u
  ]
  (
    {\xs*(\Rmaj
      -\rmin*\q/sqrt(\q*\q+\s*\s*sin(\u)*sin(\u)))*cos(\u)},
    {\ys*(\q*\Rmaj*sin(\u)
      -\rmin*sin(\u)/sqrt(\q*\q+\s*\s*sin(\u)*sin(\u)))}
  )
  -- cycle;

\def\vfield{-12}

\draw[
  orbitback,
  postaction={decorate},
  decoration={
    markings,
    mark=between positions .09 and .92 step .155 with {
      \arrow{Stealth[length=3.0mm,width=2.0mm]}
    }
  }
]
  plot[
    domain=0:180,
    samples=120,
    variable=\u
  ]
  (
    {\xs*(\Rmaj+\rmin*cos(\vfield))*cos(\u)},
    {\ys*(\q*(\Rmaj+\rmin*cos(\vfield))*sin(\u)
      +\s*\rmin*sin(\vfield))}
  );

\draw[
  orbitfront,
  postaction={decorate},
  decoration={
    markings,
    mark=between positions .07 and .93 step .135 with {
      \arrow{Stealth[length=3.2mm,width=2.15mm]}
    }
  }
]
  plot[
    domain=180:360,
    samples=120,
    variable=\u
  ]
  (
    {\xs*(\Rmaj+\rmin*cos(\vfield))*cos(\u)},
    {\ys*(\q*(\Rmaj+\rmin*cos(\vfield))*sin(\u)
      +\s*\rmin*sin(\vfield))}
  );

\end{tikzpicture}%
}
\end{figure}
\FloatBarrier

Then 

\begin{align*}
    \chi(\T^2) &= \# \T \text{ -fixed points }\\
    &= 0
\end{align*} \end{ex}

\begin{remark} The fact that $\T^2$ is a compact, connected Lie group over $\R$ implies that $\chi(\T^2)=0$. \end{remark}

\textbf{Returning to Example 2}:\\

If we view $\PP^1_\C$ as the set of lines through the origin in $\C^2$, then 

\begin{align*} 
    \PP^1_\C &= \faktor{\C^2 \setminus \{(0,0)\} }{\C^*}
\end{align*} where $v \sim w$ if $\exists \lambda\in \C^*$ s.t. $v = \lambda \cdot w$ for $v,w\in \C^2$. Then $\PP^1$ admits a $\T=\C^*$ action via

\begin{align*}
    \lambda\cdot [x_0:x_1] &= [\lambda x_0: \lambda x_1] \\
    &= [x_0: \lambda x_1]
\end{align*} for $\lambda \in \C^*$, with fixed points $pt_1 = [1:0]$ and $pt_2=[0:1]$. \\

\subsection{Schubert Cell Decomposition}
\label{schubcell}

Here, we start with a question: Which subsets of $\PP^1$ flow to each $\T$-fixed point in the limit: \[ \lim_{\lambda \rightarrow \infty} \lambda \cdot z = pt_1  \hspace{1cm} ?\]

\textbf{A}: In $U_0 = \{x_0=0\}$ with coordinate $\frac{x_1}{x_0}=z$ and the action $\lambda \cdot z = \lambda z$, we have

\begin{align*}
    \lim_{\lambda\rightarrow \infty} \lambda \cdot z = [1:0] &\iff z=0 ,\\
    \lim_{\lambda \rightarrow \infty} \lambda \cdot z = [0:1] &\iff z\neq 0. 
\end{align*} This also recovers the \textbf{cellular decomposition of Schubert cells} 

\FloatBarrier
\begin{figure}[h]
\centering
\resizebox{0.60\textwidth}{!}{%
\begin{tikzpicture}[
  line cap=round,
  line join=round,
  frontflow/.style={
    draw=blue!78!black,
    line width=1.25pt,
    postaction={decorate},
    decoration={
      markings,
      mark=at position .38 with
        {\arrow{Stealth[length=2.8mm,width=1.9mm]}},
      mark=at position .70 with
        {\arrow{Stealth[length=2.8mm,width=1.9mm]}}
    }
  },
  backflow/.style={
    draw=blue!42,
    line width=1.05pt,
    dash pattern=on 4.5pt off 3.5pt,
    postaction={decorate},
    decoration={
      markings,
      mark=at position .39 with
        {\arrow{Stealth[length=2.6mm,width=1.75mm]}},
      mark=at position .69 with
        {\arrow{Stealth[length=2.6mm,width=1.75mm]}}
    }
  }
]

\def\Rsphere{3.00}

\shade[
  ball color=gray!8
] (0,0) circle (\Rsphere);

\begin{scope}
  \clip (0,0) circle (\Rsphere);
  \foreach \a in {-54,-18,18,54}{
    \draw[backflow]
      plot[
        domain=-84:84,
        samples=90,
        variable=\t
      ]
      ({\Rsphere*cos(\t)*sin(\a)},
       {\Rsphere*sin(\t)});
  }
\end{scope}

\draw[
  gray!55,
  line width=.60pt,
  dash pattern=on 3.4pt off 2.8pt
]
  (-\Rsphere,0)
  arc[
    start angle=180,
    end angle=0,
    x radius=\Rsphere,
    y radius=.52
  ];

\begin{scope}
  \clip (0,0) circle (\Rsphere);
  \foreach \a in {-72,-36,0,36,72}{
    \draw[frontflow]
      plot[
        domain=-84:84,
        samples=90,
        variable=\t
      ]
      ({\Rsphere*cos(\t)*sin(\a)},
       {\Rsphere*sin(\t)});
  }
\end{scope}

\draw[
  gray!72,
  line width=.78pt
]
  (-\Rsphere,0)
  arc[
    start angle=180,
    end angle=360,
    x radius=\Rsphere,
    y radius=.52
  ];

\draw[
  gray!78,
  line width=.85pt
] (0,0) circle (\Rsphere);

\fill[blue!82!black] (0,-\Rsphere) circle (.105);
\node[
  below=7pt,
  text=blue!72!black
] at (0,-\Rsphere)
  {$\mathrm{pt}_1=[1\!:\!0]$};

\draw[
  blue!68!black,
  line width=.72pt,
  -{Stealth[length=2.5mm,width=1.7mm]}
]
  (4.45,.15)
  to[out=165,in=8]
  (2.84,.48);

\node[
  anchor=west,
  align=left,
  text=blue!68!black
] at (4.52,.08)
  {$B^2$\\[-.15em]
   \footnotesize $2$--cell};

\end{tikzpicture}%
}
\end{figure}
\FloatBarrier

as before.

\begin{remark} Restricting $\T$ to an $S^1=U(1)$ action recovers the $S^1$ action on $S^2$ via rotations. \end{remark} 

\textbf{Another way to see that $\chi(\PP^1)=2$} is to view $\PP^1_\C$ as the Grassmanian of 1-dim subspaces of $\C^2$: \[ \PP^1 = Gr_\C(1,2) \] 

admits a natural action of $GL(2,\C) \supset \T = \{\text{ invertible diagonal matrices} \}$. As a set, points of $Gr(1,2)$ are given in homogeneous coordinates by \\

$1\times 2$ matrices $[b_1, b_2]$ up to multiplication on the left by an invertible $1\times 1$ matrix\\

Here: The $1\times 2$ matrix represents a basis of a 1-dimensional subspace $\Lambda \subset \C^2$. Let $\T$ act on $Gr_\C(1,2)$ via

\begin{align*}
    \left[ \begin{matrix} b_1 & b_2 \end{matrix} \right]\left[ \begin{matrix} \lambda & 0 \\ 0 & \lambda^2 \end{matrix}\right]
\end{align*} 

with weights of $\C^*$ action $(1,2)$. Then $\T$-fixed 1-planes of $\C^2$ (i.e., $\T$-fixed points of $Gr(1,2)$ ) are those 1-planes admitting basis of the form 

\[ \{e_1\} \hspace{1cm} \text{ or } \hspace{1cm} \{e_2\} \] 

\textbf{Again, since there are 2 $\T$-fixed points, we see that $\chi(Gr_\C(1,2))=2$}. 

In the limit, the subset of $Gr_\C(1,2)$ flowing to $\{e_2\}$ are lines with non-zero projection to $<e_2>$, i.e.: those lines which admit a basis vector $ae_1 + e_2$ (so this subset is 1-dimensional over $\C$).\\

The subset of $Gr(1,2)$ flowing to $\{e_1\}$ are those $1-$planes admitting basis of the form $b=e_1$, which is just a point in $Gr_\C(1,2)$. 

\subsection{ $\PP^n$ (Euler characteristic and Schubert cell decomposition)}
Generalizing $\PP^1$, we have $\PP^n = \faktor{\C^{n+1}\setminus(0,0,\dots, 0) }{\C^*}$. Letting $\T = \C^* \curvearrowright \PP^n$ via 

\begin{align*} \lambda \cdot [x_0: x_1: \cdots : x_n] &= [\lambda x_0 : \lambda^2 x_1: \cdots : \lambda^{n+1}x_n]\\
&= [x_0: \lambda x_1: \cdots : \lambda^n x_n] \end{align*} we see that there are $n+1 \text{ }\T$-fixed points \[ [1:0:\cdots 0], [0:1:0:\cdots 0], \dots [0:\cdots 0:1] \] so $\chi(\PP^n)=n+1$. 

We can also see the Schubert cell decomposition from this $\T$-action.\\

\textbf{Q}: What flows to $e_i = [0:0:\cdots 0:1:0:\cdots 0]$ in the limit, where $e_i$ has a $1$ in the ith position?\\

\textbf{A}: In $U_i = \{x_i\neq 0 \}$ with coordinates $(\frac{x_0}{x_n}, \cdots, \hat{\frac{x_i}{x_n}}, \cdots \frac{x_n}{x_i} ) = (w_1, \dots, w_n)$, we have that $pt_i$ is the origin in $U_i$, and 

\begin{align*}
    \lambda \cdot  [x_0: \cdots :x_n] &= [x_0: \lambda x_1: \cdots :\lambda^n x_n ] &\text{ gives }\\
    \lambda \cdot (w_1, \dots, w_n) &= \lambda \cdot \left( \frac{ x_0}{x_i}, \cdots, \frac{x_{i-1}}{x_i}, \frac{x_{i+1}}{x_i}, \dots, \frac{x_n}{x_i} \right)\\
    &= \left( \frac{x_0}{\lambda^i x_i}, \frac{\lambda x_1}{\lambda^i x_i}, \dots, \frac{ \lambda^n x_n}{\lambda^i x_i} \right)
\end{align*} gives 

\begin{align*}
    &\lim_{\lambda \rightarrow \infty} \lambda \cdot w = pt_i = (0,0,\dots, 0) \text{ in }U_i \\[.2cm]
    \iff &\frac{x_{i+1}}{x_i} = w_i = 0, \frac{x_{i+2}}{x_i} = w_{i+1} = 0, \dots, \frac{x_n}{x_i} = w_n=0 \end{align*} 

    gives an $i$-dimensional (over $\C$) Schubert cell that flows to $pt_i$ in the limit. To clarify what happens for $pt_0$ and $pt_n$: \\
    \begin{itemize}
        \item Only $pt_0 = [1:0:\cdots :0]$ flows to $[1:0:\cdots 0]$ in the limit, and \\
        \item $\PP^n \setminus \{pt_0, \dots, pt_{n-1} \}$ flows to $pt_n = [0:0:\cdots :1]$ in the limit gives an $n$-dimensional cell (over $\C$). 
    \end{itemize}

So we see that $\PP^n$ has one Schubert cell of $\C$-dimension $k$ (so $\R$-dimension $2k$) for each $k, 0\leq k\leq n$. 

\subsection{ $Gr_\C(k,n)$ } 

$Gr_\C(k,n)$ admits a natural action of $GL(n,\C) \supset T = \{\text{ invertible diagonal }n\times n \text{ matrices }\}$. Here, $Gr(k,n)$ can be viewed as a manifold (or variety) of $k$-dimensional subspaces of $\C^n$ (or $\A^n)$. \\

$Gr(k,n)$ has homogeneous coordinates: \begin{itemize} \item List a $k\times n$ matrix giving basis for a $k$-plane $\Lambda$ with basis vectors as \textbf{rows}. \\
\item Consider this $k\times n$ matrix up to equivalence by action of $GL(k,\C)$ on the left:

\begin{align*}
    \left( \begin{matrix} & & \\ & \text{ k }\times \text{k invt'ble}  & \\ & & \end{matrix} \right) \cdot \left(\begin{matrix}  a_{11} & \cdots & a_{1n}\\ \vdots & & \vdots \\ a_{k,1} & \cdots & a_{k,n} \end{matrix}\right) &= \left( \begin{matrix} b_{11} & \cdots & b_{1n}\\ \vdots & & \vdots \\ b_{k1} & \cdots & b_{kn} \end{matrix} \right)
\end{align*}
implies \[  \left(\begin{matrix}  a_{11} & \cdots & a_{1n}\\ \vdots & & \vdots \\ a_{k,1} & \cdots & a_{k,n} \end{matrix}\right) \sim \left( \begin{matrix} b_{11} & \cdots & b_{1n}\\ \vdots & & \vdots \\ b_{k1} & \cdots & b_{kn} \end{matrix} \right) \] 

in homogeneous coordinates. \end{itemize} Now $Gr_\C(k,n) \curvearrowleft \T$ via diagonal invertible matrices on the right: 

(By abuse, I'll say): Let $\T=\C^*$ act on $Gr(k,n)$ with weight $(1,2,\cdots, n)$: For $\lambda\in \C^*$ and $\Lambda \in Gr(k,n)$, 

\[ \lambda \cdot \Lambda = \left(\begin{matrix}  a_{11} & \cdots & a_{1n}\\ \vdots & & \vdots \\ a_{k,1} & \cdots & a_{k,n} \end{matrix}\right) \cdot \left( \begin{matrix} \lambda & & & \\ & \lambda^2 & &  \\ & & \ddots  & \\ & & &  \lambda^n \end{matrix} \right) \]

Then the only $\T$-fixed points are those admitting a basis which I'll denote $\{e_I\}$ for $I\subseteq \{1,2,\dots, n\}$ with $|I|=k$. This gives \[ \chi(\text{Gr}_\C(k,n)) = \binom{n}{k}. \]

To see the Schubert cell decomposition, it's helpful to give an example first: We'll consider $Gr_\C(2,4)$. \\

Using the torus action $Gr(2,4) \curvearrowleft \T$ from $\C^n \curvearrowleft GL(n,\C)$ and $\T = \{\text{invertible diagonal }n\times n \text{ matrices } \} \subset GL(n,\C)$, via 

\[ \lambda \cdot \Lambda = \left( \begin{matrix} a_1 & a_2 & a_3 & a_4 \\ b_1 & b_2 & b_3 & b_4 \end{matrix}\right) \cdot \left( \begin{matrix} \lambda & & & \\ & \lambda^2 & & \\ & & \lambda^3 & \\ & & & \lambda^4 \end{matrix}\right) \] 

for $\{v = \sum a_i e_i, w= \sum b_i e_i\}$ a basis for $\Lambda$, we have $\chi(Gr_\C(k,n))=\binom{n}{k}$ gives $\binom{4}{2}=6$ torus-fixed points 

\[ < e_1, e_2>, <e_1,e_3>, <e_1,e_4>, <e_2,e_3>, <e_2,e_4>, <e_3,e_4> = \{pt_1, \dots pt_6\} \]

which we'll write as

\[ \{1,2\}, \{1,3\}, \{1,4\}, \{2,3\}, \{2,4\}, \{3,4\}. \]

To get a Schubert cell decomposition of $Gr(2,4)$, we consider the subset of $Gr(k,n)$ which flows to $pt_i$ in the limit: $\lim_{\lambda\rightarrow\infty} \lambda \cdot x = pt_i$ gives

\begin{itemize}
    \item $\lim_{\lambda\rightarrow\infty} \lambda\cdot x = \{1,2\} \iff \Lambda$ admits a basis of the form $\faktor{ \left( \begin{matrix} \_ & * & \cdot & \cdot \\ * & \cdot & \cdot & \cdot \end{matrix}\right)}{GL(2,\C)}$ where $*$ denotes a \textbf{non-zero} entry and $\_$ can be $0$, gives a $3-3=0$-dimensional cell. 
    \item $\lim_{\lambda\rightarrow \infty} \lambda\cdot x = \{1,3\} \iff  \Lambda$ admits a basis of the form $\faktor{ \left( \begin{matrix} \_ & \_ & * & \cdot \\ * & \cdot & \cdot & \cdot \end{matrix}\right)}{GL(2,\C)}$ gives $4-3=1$ dimensional cell\\
    \item $\lim_{\lambda\rightarrow\infty} \lambda\cdot x = \{1,4\} \iff \faktor{ \left( \begin{matrix} \_ & \_ & \_ & * \\ * & \cdot & \cdot & \cdot \end{matrix}\right)}{GL(2,\C)}$ gives $5-3=2$ dimensional cell\\
    \item $\lim_{\lambda\rightarrow\infty}\lambda\cdot x  = 
    \{2,3\} \iff \faktor{ \left( \begin{matrix} \_ & \_ & * & \cdot \\ \_ & * & \cdot & \cdot \end{matrix}\right)}{GL(2,\C)}$ 
    gives a $5-3=2$ dimensional cell. \\
    \item $\lim_{\lambda\rightarrow\infty} \lambda\cdot x = \{2,4\} \iff \faktor{ \left( \begin{matrix} \_ & \_ & \_ & * \\ \_ & * & \cdot & \cdot \end{matrix}\right)}{GL(2,\C)}$ gives a $6-3=3$-dimensional cell. \\
    \item $\lim_{\lambda\rightarrow\infty} \lambda \cdot x = \{3,4\} \iff \faktor{ \left( \begin{matrix} \_ & \_ & \_ & * \\ \_ & \_ & * & \cdot \end{matrix}\right)}{GL(2,\C)}$ gives a $7-3=4$-dimensional cell. This gives the Betti numbers \[ b_0=1, b_2=1, b_4=2, b_6=1, b_8=1, b_i=0 \text{ for }i\text{ odd. }\] over $\R$. 

    \begin{remark} This agrees with the general formula \[ b_{2i} = \lambda_{n,k}(i) \] where $\lambda_{n,k}(i)$ is the number of partitions of $i$ into $\leq n-k$ parts, each of size $\leq k$. The above Betti numbers correspond to the partitions $\{0=0\}, \{1=1\}, \{2=2+0=1+1\}, \{3=2+1\}, \{4=2+2\}$, respectively. \end{remark}
\end{itemize}

A related exercise here comes from \cite{milnor1974characteristic}.

\begin{exc} Show that the \# of $r$-cells in $Gr_\C(k,n)$ is equal to the \# of $r-$cells in $Gr_\C(n-k,n)$. \end{exc}

\begin{proof} First, we ask how to get an $r$-dimensional cell: $\dim Gr(k,n) = k(n-k) \implies 0\leq r \leq k(n-k)$. To fix a convention, let $k <  n-k$. Next, let $\Lambda \in Gr(n,k)$ live in an $r-$dimensional cell $e_I^r$ with $|I|=k$, flowing to a given $\T$-fixed point. Write $I = \{i_1>i_2> \cdots >i_k\}$. Then $\Lambda$ admits a basis (after row-reducing) with $1$ in pivot columns $I$ (can clear 0's above and $0$ in all entries for columns greater than $I$. This just gives anti-diagonal of $1's$ when restricting to columns in $I$ (or, if we like, we can have the $k \times k$ identity on these columns instead). In row $1$, we have $1$ fixed in column $i_1$ with $0$ in all other $I$ columns, and all other columns before $i_1$ outside of $I$ can be anything. This gives dimension $i_1 - k$ from row $1$. Row 2 has a $0$ in column $i_1$ and a $1$ in column $i_2$ with $i_2 - (k-1)$ degrees of freedom. And so on, until in row $k$ we have $i_k - 1$ degrees of freedom. We can write this as (here, superscripts denote row, and subscripts denote column)

\[ \left( \begin{array}{ccccccc|c} \cdots a^1_{i_{k}-1} & 0 & \cdots & 0 & \cdots & 1^{i_1} & \cdots & 0  \\ \cdots a^2_{i_{k-1}}  & 0 & \cdots & 1^{i_2} & \cdots  & 0 & \cdots& 0 \\
\vdots &  \vdots & \cdots  & \vdots &  & \vdots &   \cdots & 0  \\
\cdots a^k_{i_{k}-1} & \cdots  1^{i_k} & \cdots & 0 & \cdots &  0 & \cdots& 0 \end{array} \right) \] 

where the $\_ $ indicates any value (in columns outside of $I$ and below $i_k$), $1$'s appear exactly in $I$, and all entries above $i_1$ are $0$. 

Thus \begin{align*} r &= (i_1 - k) + (i_2 - (k-1)) + \cdots + (i_k -1 )  \\
&= i_1 + i_2 + \cdots + i_k - (1+2+\cdots + k) \\
&= i_1 + i_2 + \cdots + i_k - \frac{k(k+1)}{2}. \end{align*} 

So again, let $\Lambda$ live in an $r$-dimensional Schubert cell $e_I^r$. We saw that $\Lambda$ admits a basis as above. Since the orthogonal complement to the row space is the null space of the matrix $\Lambda$ (i.e., the kernel), $\Lambda^\perp$ admits a basis of the form: standard basis column vectors for columns in $[n]\setminus I$ (so that we have $1$'s on the \textbf{antidiagonal} restricted to pivot columns) with the entries in each row from $I$ determined by living in the kernel of $\Lambda$ (i.e., dot product with each row of $\Lambda$ should be $0$). Thus $\Lambda^\perp$ can be written

\[ \Lambda^\perp  = \left( \begin{array}{ccccc|cc}
0 &0 &0 &0 &0 & 0 & 1 \\ 0 &0& 0& 0& 0 & 1 & 0  \\ 0 &0 &0 & 0 & \cdots -a^1_{i_1-1} & 0 & 0  \\  \vdots & \vdots & \vdots & \vdots & \vdots & \hspace{.2cm}  \vdots  \\
0 & 0 & 1 & -a^{k-1}_{([n]\setminus I)_2} & \cdots  -a^1_{([n]\setminus I)_2} & 0 & 0 \\
1 & -a^k_{([n] \setminus I)_1}\cdots & 0   & -a^{k-1}_{([n]\setminus I)_1} \cdots & -a^1_{([n]\setminus I)_1} & 0 & 0  \end{array}\right)  \]

where $([n]\setminus I)_1$ is the first column from the left in $[n]-I$, and here we have basis column vectors in the $n-k$ pivot columns $[n]\setminus I$. 
The variable entries $-a^i_j \in \C$ now correspond to exactly an $r$ $\C-$dimensional Schubert cell of $Gr_\C(n-k,n)$ under the $\T$-action of $\C^*$ on $Gr(n-k,n)$ corresponding to the weight $(n,n-1, \dots, 1)$:\\

\[ Gr(n-k,n) \curvearrowleft \T \text{ via } \lambda \cdot \Lambda^\perp = \left( \begin{matrix} a_{1,1} & \cdots & a_{1,n} \\ \vdots & \ddots & \vdots \\ a_{{n-k},1} & \cdots & a_{(n-k),n} \end{matrix} \right) \cdot \left( \begin{matrix} \lambda^n & & & \\ & \lambda^{n-1} & & \\ & & \ddots & \\ & & & \lambda^1 \end{matrix} \right) \] 

It is maybe helpful to clarify the above general statement with examples to see the correspondence between $\Lambda, \Lambda^\perp$, and Schubert cells. Below, we sometimes use the convention that $I$ has $k\times k$ identity matrix (rather than antidiagonal of $1$'s). 

\subsubsection{$Gr(1,3)$}
For $X = Gr_\C(1,3)$ with $\T$-action written as \[ \left( \begin{matrix}  b_1 & b_2 & b_3 \end{matrix} \right) \left( \begin{matrix} \lambda & & \\ & \lambda^2 & \\ & & \lambda^3 \end{matrix} \right), \] we have $\T$ fixed points $\{(1,0,0), (0,1,0), (0,0,1)\}$ which I'll denote $\{e_1, e_2, e_3\}$. Then the $0$-dimensional Schubert cell $e_I^r$ of all points in $X$ flowing to $e_1$ in the limit is just $(1,0,0)$. The set of all $\Lambda^\perp$ for $\Lambda \in e^r_I$ is just $\left( \begin{matrix} 0 & 1 & 0 \\ 0 & 0 & 1 \end{matrix} \right)$. This is also the $0$-dimensional cell $e^r_{[n]\setminus I} \subset Gr(n-k,n)$ which flow to the $\T-$fixed point $\{2,3\}$ under the $\T$-action weight $(3,2,1)$. \\

The $1 \C$-dimensional cell $e_I^1$ flowing to $e_2$ is all $\Lambda \in Gr(1,3)$ admitting a basis of the form $(a, 1, 0)$ with $a\in \C$. Then $\{\Lambda^\perp \text{ }|\text{ } \Lambda \in e^1_I\}$ is the set $\left\{ \left( \begin{matrix} 0 & 0 & 1 \\ 1 & -a & 0 \end{matrix}\right) = \left( \begin{matrix} 1 & -a & 0 \\ 0 & 0 & 1 \end{matrix}\right) \right\} $, which is the same as the cell $e^1_{[n]\setminus 2} \subset Gr(2,3)$ flowing to $\{1,3\}$, which we'd write as \[ \left( \begin{matrix} 1 & A & 0 \\ 0 & 0 & 1 \end{matrix}\right) \] for $A\in \C$.   \\

Similarly, the $2 \C$-dimensional cell flowing to $e_3$ is given by $\{ (a, b, 1 ) \text{ }|\text{ } a,b\in \C\}$ has the set of perpendicular subspaces $\{ \left( \begin{matrix} 1 & 0 & -a \\ 0 & 1 & -b \end{matrix} \right) \text{ }|\text{ } a,b \in \C\}$ which is exactly the $2 \C$-dimensional cell of $Gr(2,3)$ flowing to $\{1,2\}$ under the $\T$-action with weight $(3,2,1)$. \\

\subsubsection{$Gr(2,5)$}

$X = Gr(2,5)$ has $\T-$fixed points $\{ \{1,2\},\{1,3\}, \{1,4\}, \{1,5\}, \{2,3\}, \{2,4\}, \{2,5\}, \{3,4\}, \{3,5\}, \{4,5\} \}$.

For $I=\{1,2\}$, we have $e^0_I \subset Gr(k,n)$ given by \[ \left\{ \Lambda = \left( \begin{matrix} 1 & 0 & 0 & 0 & 0 \\ 0 & 1 & 0 & 0 & 0 \end{matrix}\right) \right\} \] which has perpendicular subspaces

     \[  \left\{ \Lambda^\perp = \left( \begin{matrix} 0 &0 &1& 0& 0 \\ 0& 0 &0& 1 &0 \\ 0& 0& 0& 0& 1 \end{matrix} \right) \right\} \] giving the $0$ cell of $Gr(n-k,n)$ flowing to $\{3,4,5\}$.

For $I = \{1,3\}$, we have the $1 \text{ }\C-$dimensional cell \[ \left\{ \Lambda = \left( \begin{matrix} 1 &0 &0& 0& 0 \\ 0 &a& 1 &0& 0 \end{matrix} \right) \right\} \] with perpendicular subspaces \[ \left\{ \left( \begin{matrix} 0 & 1 & -a & 0 & 0 \\ 0 & 0 & 0 & 1 & 0 \\ 0 & 0 & 0 & 0 & 1 \end{matrix}\right) \right\} \] which give the $1 \C$-dimensional Schubert cell of $Gr(n-k,n)$ flowing to $\{2,4,5\}$. \\

Continuing similarly, for $I=\{3,4\}$, we have 

\[ \left\{ \Lambda = \left(\begin{matrix} a & b & 1 & 0 & 0 \\ c & d &0 &1 &0  \end{matrix}\right) \right\} \leftrightarrow \left\{ \Lambda^\perp \left( \begin{matrix} 1 & 0 & -a & -c & 0 \\ 0 & 1 & -b & -d & 0 \\ 0 &0 &0 &0 &1 \end{matrix} \right) \right\}  \]

which is exactly the Schubert cell of $Gr(n-k,n)$ flowing to $\{ 1,2,5\}$. 

Here we can see that each parameter of the LHS appears exactly once on the right-hand side, and a transpose occurs (along with adding in $-$ sign). \\

For $I=\{4,5\}$, we have \[ \left\{ \Lambda = \left( \begin{matrix} a &b& c& 1& 0 \\ d &e &f &0 &1 \end{matrix}\right) \right\} \leftrightarrow \left\{ \Lambda^\perp = \left( \begin{matrix} 1 &0 &0 &-a& -d \\ 0 &1 &0 &-b& -e \\ 0 &0 &1 &-c& -f \end{matrix} \right) \right\}  \] giving a $6 \text{ }\C-$dimensional cell on each side. \\

\subsubsection{$Gr_\C(2,6)$} 

To see more of the phenomena of the way that non-zero entries intermingle with the columns of $I$, we see that in $I = \{2,3\}$, for instance, we have that $e^2_{\{2,3\}}$, the $2 \text{ }\C-$dimensional Schubert cell flowing to $I$, can be written \[ \left\{ \Lambda = \left( \begin{matrix} a & 1 & 0 & 0 & 0 & 0 \\ b & 0 & 1 & 0 & 0 & 0 \end{matrix} \right) \right\}  \]

so that $\{ \Lambda^\perp \text{ }|\text{ } \Lambda \in e^2_{\{2,3\}} \subset Gr(2,6)\}$ can be written as \[ \left( \begin{matrix} 1 & -a & -b & 0 & 0 & 0 \\ 0 &0 &0 &1 &0 &0 \\ 0& 0& 0 &0 &1 &0 \\ 0 &0 &0 &0 &0 &1 \end{matrix}\right) \] which is exactly the $2 \text{ }\C-$dimensional Schubert cell of $Gr(4,6)$ flowing to $\{1,4,5,6\} = [n]\setminus I$ in the limit, with $\T$-action weight $(6,5,4,3,2,1)$. \\

Similarly (to illustrate this phenomenon without listing all instances): for $I=\{2,5\}$, we have $e^4_{\{2,5\}}$ consists of \[ \left\{ \Lambda =  \left( \begin{matrix} a &1 &0 &0 &0 &0\\ d &0 &b &c &1 &0 \end{matrix} \right) \right\} \]

with perpendicular complements \[  \left\{ \Lambda^\perp = \left( \begin{matrix} 1 & -a & 0 &0 &-d &0 \\ 0 &0& 1 &0& -b &0 \\ 0 &0& 0& 1 &-c &0 \\ 0& 0& 0 &0 &0 &1 \end{matrix} \right) \right\}.  \] 

This is exactly the $4 \text{ }\C-$dimensional cell of $Gr(n-k,n)$ flowing to $\{1,3,4,6\}$ is 
\[ e^4_{[n]\setminus I} = \left\{ \left( \begin{matrix} 1 & D & 0 &0 &C &0 \\ 0 &0 &1 &0 &B &0 \\ 0& 0 &0 &1 &A &0 \\ 0& 0& 0& 0& 0& 1 \end{matrix} \right)\;\middle|\; A,B,C,D \in \C  \right\} \subset Gr(n-k,n). \]


\end{proof}

\subsection{Complete flag variety}

Let $\mathcal{F}l_n$= the complete flag variety, consisting of all complete flags in $\C^n$: $\{0 =V_0 \subset V_1 \subset \cdots \subset V_n = \C^n \}$. 

That is, (generalizing $Gr_\C(k,n)$ of $k$-dimensional subspaces of $\C^n$, we consider the following set of choices:

\begin{itemize}
    \item Choose a $0$-dimensional subspace of $\C^n$, $\{0\}$. \\
    \item Choose a $1$-dimensional subspace of $\C^n, \left< v_1 \right> = V_1$ (for $v_1\neq 0$). \\
    \item Choose a $d$-dimensional subspace of $\C^n$ containing $V_1$, i.e., $V_2 = \left< v_1, v_2 \right>$ for $v_2$ linearly independent from $v_1$.\\
    \vdots \\
    \item Choose an $n-1$-dimensional subspace $V_{n-1}= \left< v_1, \dots, v_{n-1} \right>$\\
    \item Choose an $n$-dimensional subspace of $\C^n$, which must be all of $\C^n$. That is $\left< v_1, v_2, \dots, v_n \right> = V_n = \C^n$. 
\end{itemize}

Then $Fl_n$ admits a natural $GL(n,\C) \supset \T = \{\text{ invertible }n\times n \text{ matrices }\}$ action by matrix multiplication on the right of \[ \left[ \begin{matrix} \leftarrow v_1 \rightarrow \\ \leftarrow v_2 \rightarrow \\ \vdots \\ \leftarrow v_n \rightarrow \end{matrix} \right] \left[ \begin{matrix} \lambda & & & \\ & \lambda^2 & & \\ & & \ddots & \\ & & & \lambda^n \end{matrix} \right] \]

so that the only $\T$-fixed points of $\mathcal{F}l_n$ are those complete flags $\{ \{0\}=V_0 \subset V_1 \subset V_2 \subset \cdots \subset V_n = \C^n \}$ such that each $V_i$ admits a basis which is a subset of $\{e_1, \dots, e_n \}$. \\

\[ \implies \chi(\mathcal{F}l_n) = n! \] 

\begin{exc} Show $\dim_\C \mathcal{F}l_n = \binom{n}{2}$ \end{exc} \textbf{Hint}: $\binom{n}{2} = 1 + 2 + \cdots + n-1$ 

\begin{exc} (Harder): Describe the Schubert cell decomposition of $Fl_n$. \end{exc} 

\subsection{Localization argument for 27 lines on a nonsingular cubic surface $X$}

\label{sec: loc27lines}
Here we follow the conventions and background from \cite{MSTextbook} and \cite{coxkatxMSAG}. Recall from Section~\ref{chernclassCubicSurface} that a line $\Lambda \subset \PP^3$ is contained in $X \iff |X\cap \Lambda| \geq 4 \text{ points}$ by Bezout's theorem. Recall that the zero locus $Z(\sigma_f)$, denoted $F_1(X)$, has codimension $4 = \text{rank }(Sym^3(\mathcal{S}^*))$ in $\text{Gr}(2,4)$. To evaluate 
\[ \int_{\text{Gr}(2,\A^4)} c_4 \text{Sym}^3(\mathcal{S}^*) \]

we can use the action of $\mathbb{T}=(\C^*)^4 \curvearrowright V = \C^4$ via \[ (\lambda_1, \dots, \lambda_4) \cdot (x_1, \dots, x_4) = (\lambda_1^{-1} x_1, \dots, \lambda_4^{-1}x_4), \] which induces an action of $(\C^*)^4 \curvearrowright Gr(2,4)$ via 

\begin{align*}
    \left( \begin{matrix} a_1 & a_2 & a_3 & a_4 \\ b_1 & b_2 & b_3 & b_4 \end{matrix}\right) \cdot \left( \begin{matrix} \lambda_1^{-1} & \cdot & \cdot & \cdot \\ \cdot & \lambda_2^{-1} & \cdot & \cdot \\ \cdot & \cdot & \lambda_3^{-1} & \cdot \\ \cdot & \cdot & \cdot & \lambda_4^{-1} \end{matrix}\right).
\end{align*}

Without loss of generality, the $2$-plane $\Lambda \subset \C^4$ can be written with $\text{Id}_{2\times 2}$ on pivot columns with $i_1 < i_2$ as, for instance:

\FloatBarrier
\begin{center}
\begin{figure}[h]
\includegraphics[trim = {0cm, 14cm, .8cm, .2cm}, clip, width=.4\textwidth]{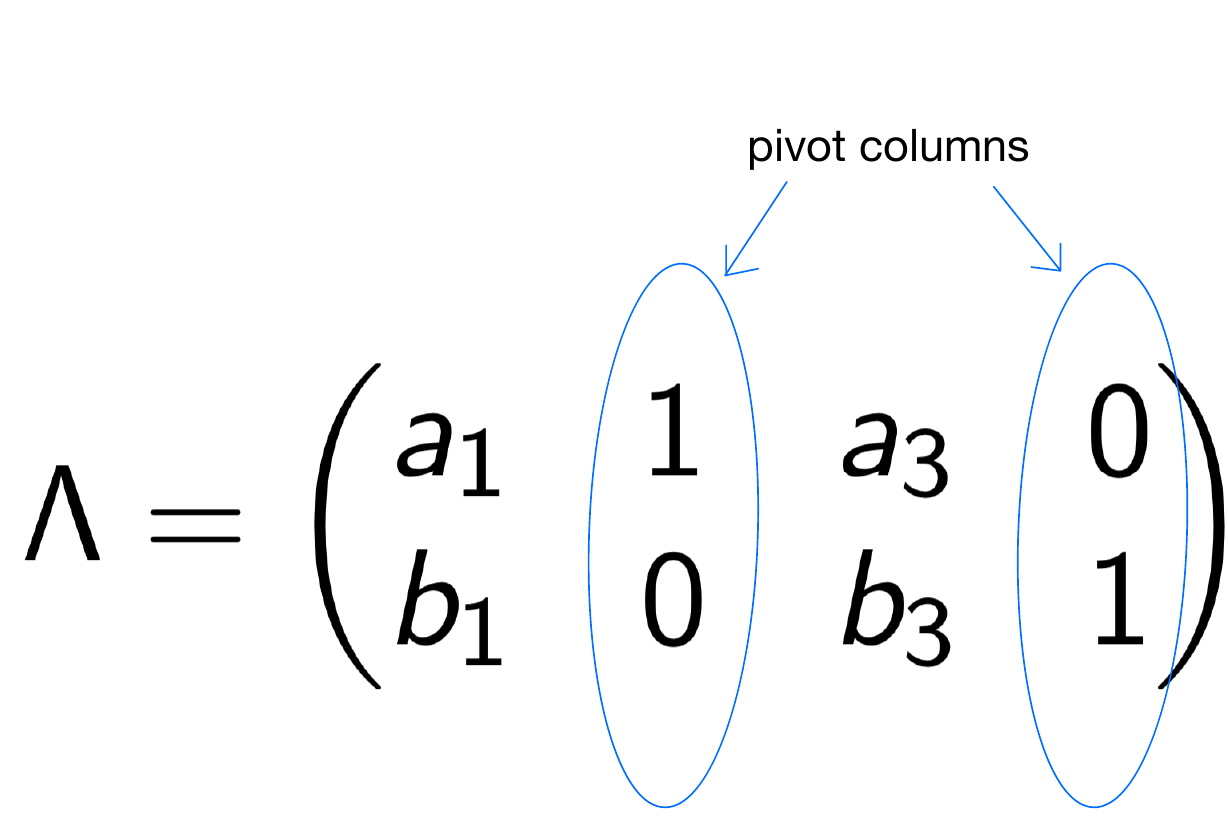}
\caption{Here $i_1=2 < i_2=4$}
\end{figure} 
\end{center} 
\FloatBarrier

To be $\mathbb{T}-$fixed, we need $a_j= b_j = 0$ for $j\not\in I$. This gives the $\binom{4}{2}=6 \text{ } \T-$fixed points, which we'll denote as \[ \{L_I\} = \{ (1 2),(1 3),(1 4), (2 3), (2 4), (3 4) \}. \]

At a given $2-$plane $\Lambda \in \text{Gr}(2,4)$, we have 

\[ N_{\faktor{L_I}{Gr(2,4)}} = T_{L_I}Gr(2,4) \] since $L_I$ is a point. To get an equivariant cohomology class of $N_{L_I}$, we decompose $T_{L_I}Gr(2,4)$ into weight-spaces as follows. Note that the tangent bundle $T Gr(2,4) \cong Hom (\mathcal{S}, \mathcal{S}^\perp) = Hom(\mathcal{S}, \faktor{\C^4}{\mathcal{S}} ). $

Viewing $L_I$ as $\{x_j = 0\text{ }|\text{ }j\not\in I\}$, we have characters of $\mathcal{S}\bigg|_{L_I}$ as \[-\rho_i \text{ for } i\in I. \] 

 So \[ \text{Hom}_{L_I}(\mathcal{S}, \mathcal{S}^\perp) \cong \mathcal{S}^* \otimes \mathcal{S}^\perp\bigg|_{L_I}  \] has characters  \[ \{\rho_i - \rho_j \text{ } |\text{ } i \in I, j\not \in I \}. \]

Therefore $T_{L_I}Gr(2,4)$ has equivariant Euler class $\prod_{i\in I, j\not\in I} \left(\lambda_i - \lambda_j\right) $ and at $L_I$: \[ \text{Euler}_\T \left(N_{\faktor{L_I}{\text{Gr}(2,4)}} \right) = \text{Euler}_\T \left( T_{L_I}(Gr(2,4)\right). \]

Next, noting that $\mathcal{S}^*\bigg|_{L_I}$ has characters $\{\rho_i\text{ }|\text{ }i\in I\}$, we have that the equivariant Euler class of $i_I^*(\text{Sym}^3(\mathcal{S}^*))$ is \[ \prod_{a=0}^3 \left( a\lambda_{i_1} + (3-a)\lambda_{i_2}  \right) \] for $I = \{i_1, i_2\}$.  This gives 

\[ \int_{\text{Gr}(2,4)}c_4 \text{Sym}^3(\mathcal{S}^*) = \sum_{|I|=2} \frac{ \prod_{a=0}^3 (a\lambda_{i_1} + (3-a)\lambda_{i_2} ) }{\prod_{i\in I, j\not\in I} (\lambda_i - \lambda_j) } = \cdots \]

\begin{align*}
\cdots = &\frac{(3\lambda_1)(2\lambda_1 + \lambda_2)(\lambda_1+2\lambda_2)(3\lambda_2)}{(\lambda_1-\lambda_3)(\lambda_1-\lambda_4)(\lambda_2-\lambda_3)(\lambda_2-\lambda_4)} + \cdots \\
&\cdots + \frac{(3\lambda_1)(2\lambda_1+\lambda_3)(\lambda_1+2\lambda_3)(3\lambda_3)}{(\lambda_1-\lambda_2)(\lambda_1-\lambda_4)(\color{red}{\lambda_3-\lambda_2}\color{black})(\lambda_3-\lambda_4)} + \cdots \\
&\cdots + \frac{(3\lambda_1)(2\lambda_1 + \lambda_4)(\lambda_1+2\lambda_4)(3\lambda_4)}{(\lambda_1-\lambda_2)(\lambda_1-\lambda_3)(\color{red}{\lambda_4-\lambda_2}\color{black})(\color{red}{\lambda_4-\lambda_3} \color{black})} \cdots \\
&\cdots + \frac{3\lambda_2(2\lambda_2+\lambda_3)(\lambda_2+2\lambda_3)(3\lambda_3)}{(\color{red}{\lambda_2-\lambda_1}\color{black})(\lambda_2-\lambda_4)(\color{red}{\lambda_3-\lambda_1}\color{black})(\lambda_3-\lambda_4)} +\cdots \\
&\cdots + \frac{3\lambda_2(2\lambda_2+\lambda_4)(\lambda_2+2\lambda_4)(3\lambda_4)}{(\color{red}{\lambda_2-\lambda_1}\color{black})(\lambda_2-\lambda_3)(\color{red}{\lambda_4-\lambda_1}\color{black})(\color{red}{\lambda_4-\lambda_3\color{black})}} +\cdots \\
&\cdots + \frac{3\lambda_3(2\lambda_3+\lambda_4)(\lambda_3+2\lambda_4)(3\lambda_4)}{(\color{red}{\lambda_3-\lambda_1}\color{black})(\color{red}{\lambda_3-\lambda_2}\color{black})(\color{red}{\lambda_4-\lambda_1}\color{black})(\color{red}{\lambda_4-\lambda_2}\color{black})} 
\end{align*}

\begin{align*}
\cdots &= \frac{ \left(  \makecell{ \left[ 18\lambda_1^3\lambda_2 + 45 \lambda_1^2\lambda_2^2 + 18\lambda_1\lambda_2^3 \right]\cdot \left[ \left(\lambda_1-\lambda_2)(\lambda_3-\lambda_4 \right) \right] + \cdots  \\ \newline
\cdots +\color{red}{(-1)^1}\color{black} \left[ 18 \lambda_1^3\lambda_3 + 45 \lambda_1^2 \lambda_3^2 + 18 \lambda_1 \lambda_3^3 \right] \cdot \left[ (\lambda_1-\lambda_3)(\lambda_2-\lambda_4)\right] + \cdots \newline\\ 
\cdots + \color{red}{(-1)^2}\color{black} \left[ 18\lambda_1^3\lambda_4 + 45\lambda_1^2 \lambda_4^2 + 18 \lambda_1 \lambda_4^3 \right] \cdot \left[ (\lambda_1-\lambda_4)(\lambda_2-\lambda_3) \right] + \cdots \\ \newline \cdots +  \color{red}{(-1)^2} \color{black} \left[ 18 \lambda_2^3 \lambda_3 + 45 \lambda_2^2 \lambda_3^2 + 18 \lambda_2 \lambda_3^3 \right] \cdot \left[ (\lambda_1-\lambda_4)(\lambda_2-\lambda_3) \right] + \cdots \\ \newline \cdots + \color{red}(-1)^3 \color{black}\left[ 18 \lambda_2^3 \lambda_4 + 45 \lambda_2^2 \lambda_4^2 + 18\lambda_2 \lambda_4^3 \right] \cdot \left[ (\lambda_1-\lambda_3)(\lambda_2-\lambda_4) \right] + \cdots \\ \newline \cdots + \color{red}(-1)^4 \color{black} \left[ 18 \lambda_3^3 \lambda_4 + 45 \lambda_3^2 \lambda_4^2 + 18 \lambda_3 \lambda_4^3 \right] \cdot \left[ (\lambda_1 - \lambda_2 )(\lambda_3 - \lambda_4 ) \right]} \right)}   {\left[(\lambda_1-\lambda_3)(\lambda_1-\lambda_4)(\lambda_2-\lambda_3)(\lambda_2-\lambda_4)\right]\cdot \left[ (\lambda_1-\lambda_2)(\lambda_3-\lambda_4)\right]}  \\[.4cm]
    & =27. 
\end{align*} 

\FloatBarrier
\begin{figure}[h]
\includegraphics[width=.8\textwidth]{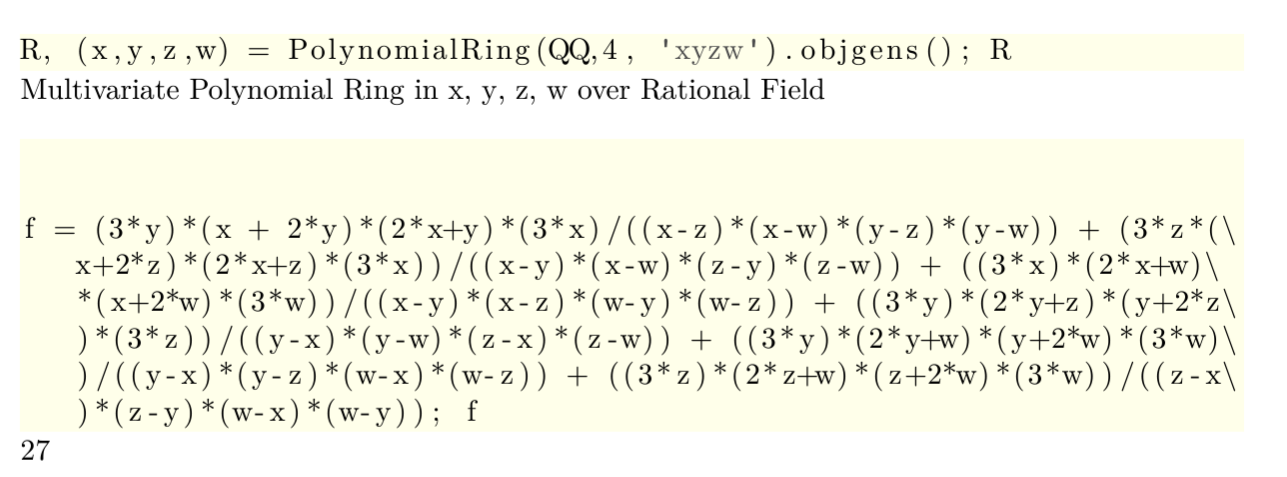}
\caption{Sage computation verifying result for 27 lines on nonsingular cubic surface using localization in equivariant cohomology}
\end{figure}
\FloatBarrier

\subsubsection{Localization argument for 2875 lines on a generic nonsingular quintic threefold $Q \subset \PP^4$}

Similarly, for 2875 lines on a generic nonsingular quintic threefold $Q\subset \PP^4$, we consider $\mathcal{S}$ as a tautological rank $2$ vector bundle on $\text{Gr}(2,5)=\text{Gr}(2,\A^5)$ and $\sigma_f \in \Gamma \text{Sym}^5(\mathcal{S}^*)$. We also note that $\Lambda \subset \PP^4$ satisfies $\Lambda \subset Q$ iff $|\Lambda \cap Q| \geq 6$, which is also the rank of $\text{Sym}^5(\mathcal{S}^*)$. Now the number of lines in $Q$ is

\begin{align*} \int_{\text{Gr}(2,5)} c_6 \text{Sym}^5(\mathcal{S}^*) &= \sum_{\makecell{|I|=2, \newline I \subset \{1,2,3,4\}}} \frac{ \prod_{a=0}^5 \left( a\lambda_{i_1} + (5-a)\lambda_{i_2}\right)}{ \prod_{\makecell{i\in I, \newline j\not\in I}} \left(\lambda_i - \lambda_j\right) }\\
    &= 2875. 
\end{align*}

\section{Main theorems of this paper}

\subsection{Description of Gromov-Witten invariants} For the purpose of exposition, we include the definition of Gromov-Witten invariants as given in \cite{coxkatxMSAG}. Let $X$ be a projective algebraic variety. Fix $\beta \in H_2(X,\Z)$ and cycles $Z_1,\dots, Z_n \subset X$. A first approximation to a Gromov-Witten invariant is given by the following remark.\\

\begin{remark} Find the structure of the set of curves $C \subset X$ of genus $g$, homology class $\beta$, with $C \cap Z_i \neq \emptyset$ for all $i$, assuming that the collection of $Z_i$ are in general position. \end{remark}

A first refinement of this is the following. 

\begin{remark} Replace curves $C \subset X$ with $n$-pointed curves (possibly reducible) $(C,p_1, \dots, p_n)$ and holomorphic maps $f:C \rightarrow X$. So consider \[ f: C\rightarrow X \text{ such that } f_*[C]=\beta \text{ and } f(p_i) \in Z_i \text{ for } i=1,\dots, n. \] 

(Up to isomorphism of maps, so consider their \textbf{moduli}). \end{remark}

We again must refine this remark, as follows, due to Kontsevich, which uses the notion of \textbf{stable map}. 

\begin{remark} An $n$-pointed \textbf{stable map} consists of a connected marked curve $(C,p_1, \dots, p_n)$ and morphism $f:C \rightarrow X$ such that: 
\begin{itemize}
    \item[(i)] The only singularities of $C$ are ordinary double points \\
    \item[(ii)] $p_1, \dots, p_n$ are distinct ordered smooth points of $C$ \\
    \item[(iii)] If $C_i$ is a component of $C$ such that $C_i \cong \PP^1$ and $f$ is constant on $C_i$, then $C_i$ has at least $3$ special (marked or nodal) points. 
    \item[(iv)] If $C$ has arithmetic genus $1$ and $n=0$, then $f$ is not constant. 
\end{itemize}

Given the first 2 conditions, (iii)-(iv) are equivalent to $(f,C,p_1,\dots, p_n)$ having only finitely many automorphisms. If $n+2g<3$, then the presence of infinitely many automorphisms implies that $\overline{M_{g,n}}$ does not exist. 
\end{remark} 

Here, a stable map is different from an $n$-pointed stable curve. Note that $\overline{M}_{g,n}$ in the sense of Deligne-Mumford exists as an orbifold of dimension $3g-3+n$ when $n+2g\geq 3$. Furthermore, $\overline{M_{g,n}}(X,\beta)$, the coarse moduli space of the stack $\overline{\mathcal{M}_{g,n}}(X,\beta))$ of genus $g$ stable maps with $n$ marked points to $X$ with $f_*[C]=\beta$ has expected dimension $(1-g)(\dim X - 3)-\int_\beta \omega_X + n$. To define a Gromov-Witten invariant, we use the natural maps

\[\begin{array}{ccccc}
 & & \overline{M}_{g,n}(X,\beta) & & \\
 & \pi_1 \swarrow  & &  \searrow \pi_2 &  \\
 X^n & & & & \overline{M}_{g,n} 
\end{array}\]

where $\pi_1( (f,C,p_1, \dots, p_n)) = (f(p_1),f(p_2),\dots, f(p_n))\in X^n$ and $\pi_2( (f,C,p_1, \dots, p_n)  )$ successively contracts non-stable components of $C$ to give a stable curve $\tilde{C}$. So $\pi_2(f,C,p_1, \dots, p_n)$ is the isomorphism class of $\tilde{C}$. We can now define Gromov-Witten classes and Gromov-Witten invariants.

\begin{definition} For $n+2g\geq 3$, we define the Gromov-Witten class \[ I_{g,n,\beta}(\alpha_1, \dots, \alpha_n) = \pi_{2!}(\pi_1^*(\alpha_1 \otimes \cdots \otimes \alpha_n)) \] where $\alpha_i$ are cohomology classes dual to the cycles $Z_i$. The right-hand side is a cohomology class of $\overline{M}_{g,n}$ of degree $2(g-1)\dim X + 2\int_\beta \omega_X + \sum_{i=1}^n \deg \alpha_i$. \end{definition}

A Gromov-Witten invariant is now given as the following integral.

\begin{definition}
    \begin{align*} \left< I_{g,n,\beta}\right>(\alpha_1,\dots, \alpha_n) &= \int_{\overline{M}_{g,n}} I_{g,n,\beta}(\alpha_1,\dots, \alpha_n) \\
    &= \int_{\overline{M}_{g,n}(X,\beta)} \pi_1^*(\alpha_1\otimes \cdots \otimes \alpha_n). \end{align*} 
\end{definition}

\begin{example} For $X$ a nonsingular cubic surface in $\PP^3$, we have that 

\[ \left< I_{0,0,\ell}\right> = 27 \] for $\ell$ the class of a line in $H_2(X,\Z)$.  \end{example}

\subsection{Genus-0 Gromov-Witten invariants for $\PP^1$}
\label{subsec: GWInvtsP1}

From the Euler sequence for $\mathcal{T}_{\PP^r}$

\[ 0 \rightarrow \mathcal{O} \rightarrow \mathcal{O}(1)^{\oplus r+1} \rightarrow \mathcal{T}_{\mathbb{P}^r} \rightarrow 0\] we get that

\begin{align*}c(\mathcal{T}_{\PP^r}) &= (1+h)^{r+1} \\ &= \sum_{k=0}^{r+1} \binom{r+1}{k} h^k\\
&= 1 + (r+1)h + \cdots + (r+1)h^r \end{align*} 

gives $c_1(T\PP^r) = (r+1)h$ and for $\beta = d\ell \in H_2(\PP^r,\Z)$, we have that $$-\int_\beta \omega_{\PP^r} = -\int_{d\ell} \omega = \int_{d\ell} c_1 T(\mathbb{P}^r) = (r+1)d.$$ Now Equation (7.33) from \cite{coxkatxMSAG} reads 

\begin{align*}
(7.33) \hspace{2cm} \sum_{i=1}^n \deg\alpha_i = 2(1-g)\dim X - 2\int_\beta \omega_X + 2(3g-3+n)\end{align*}

For $\PP^r$, this gives

\[ \sum_{i=1}^n \deg \alpha_i = 2(1-g)r + 2(r+1)d + 2(3g-3+n).\]

For $X=\PP^r, g=0$ we have

\begin{align*}
    \sum_{i=1}^n \deg\alpha_i &= 2r +2(r+1)d -6 + 2n\\
    &= 2r + 2rd + 2d -6 + 2n.
\end{align*}

For $g=0,d=0$, we have that the only non-zero invariant (up to permuting the inputs) is 

\begin{align*}
    \left<I_{0,3,0}\right>([pt], [\PP^1], [\PP^1]) &=\int_{\PP^1} [pt] \cup [\PP^1] \cup [\PP^1]\\
    &= \int_{\PP^1}[pt] = 1
\end{align*} by Equation (7.35). For $d>0$, the Fundamental Class Axiom implies that all inputs must be $([pt],[pt],\dots, [pt])$. This leaves

\begin{align*}
    \left< I_{0,n,d}\right>([pt],[pt],\dots,[pt]) &= \left( \int_\beta [pt] \right) \left< I_{0,n-1,d}\right> ([pt],[pt],\dots,[pt])\\
    &= d \left< I_{0,n-1,d}\right>([pt],[pt],\dots,[pt])
\end{align*}

by the Divisor Axiom, for $n>0$. To find $d$, we note from (7.33) that

\begin{align*}
    \sum_{i=1}^n \deg\alpha_i = 2n &= 2 + 4d - 6 + 2n\\
    &= 4d + 2n - 4
\end{align*} gives $d=1$. Now

\begin{align*}
 \left<I_{0,n,1}\right>([pt],[pt],\dots,[pt]) = \cdots &= (1)^{n-1}\cdot \left<I_{0,1,1}\right>([pt]) \text{ by the Divisor Axiom and Induction}\\
 &= (1)^n \left< I_{0,0,1} \right> = 1
 \end{align*} 

 so that $\left<I_{0,n,1}\right>([pt]^{\bullet n}) = 1$ for $n\geq 0$. 

\subsection{Quantum cohomology of $\PP^1$}

\subsubsection{Small quantum cohomology on $\PP^1$}

We fix the ordered basis $\{T_0, T_1\} = \{[\PP^1], [pt]\}$ for $H^*(\PP^1,\Q)$. The small quantum product $a*b$ for $a,b\in H^*(\PP^1,\Q)$ is defined by

\begin{align*}  a*b &= \sum_{i,j} \sum_{\beta\in H_2(\PP^1, \Z)} \left< I_{0,3,\beta} \right>(a,b,T_i)g^{ij}q^\beta T_j\\
&= \sum_i \sum_{d\geq 0} \left< I_{0,3,d} \right> (a,b,T_i) q^\beta T^i\\
&= \sum_i \sum_{d=0,1} \left< I_{0,3,d}\right> (a,b,T_i) q^{d\cdot \ell} T^i
\end{align*} 

which we extend to $H^*(\PP^1, \C)$ by linearity. The small quantum product on basis elements is given by 

\begin{align*}
    T_0 * T_0 &= \left< I_{0,3,0}\right>(T_0, T_0, T_1) q^0 T^1 = T_0,\\
    T_0*T_1 &= \left< I_{0,3,0}\right> (T_0, T_1, T_0) q^0 T^0 = T^0 = T_1\\
    T_1*T_1 &= q^\ell T^1 = q^\ell T_0
\end{align*}

\subsubsection{$QH^*(\PP^1)$: Constructing the big quantum cohomology of $\PP^1$}

Here we again follow \cite{coxkatxMSAG}. The cohomology basis for $\PP^1$ is $T_0 = [\PP^1]$ and $T_1 = [pt]$. To make $\Phi$ homogeneous, we set $\deg t_0 = -2$ and $\deg t_1 = 0$ \cite{coxkatxMSAG} (8.28). For $\beta=d\ell$ and $q^\beta = e^{2\pi i \int_\beta \omega_{\PP^1}}$, we set

\[ \deg q^\beta = 2 \int_\beta \omega_{\PP^1} = 2 \int_{d\ell} \omega_{\PP^1} =-2(2)d = -4d \]

Let $\gamma = t_0T_0 + t_1T_1$. 
Then the Gromov-Witten potential is the formal sum

\begin{align}\label{eqn: GWPotP1}
\Phi(\gamma) &= \sum_{n=0}^\infty \sum_{\beta\in H_2(\PP^1,\Z)}\frac{1}{n!}\left<I_{0,n,\beta}\right>(\gamma^n)q^\beta\\
&= \sum_{n=0}^\infty \sum_{\beta = d\ell, d\geq 0} \frac{1}{n!} \left< I_{0,n,d}\right>((t_0[\PP^1] + t_1[pt])^{\bullet n}) q^\beta \\
&= \frac{1}{3!} \cdot 3 \left<I_{0,3,0}\right> (t_1[pt],t_0[\PP^1],t_0[\PP^1])q^0 + \sum_{n\geq 0} \frac{1}{n!} \left< I_{0,n,1}\right>(pt^{\bullet n})t_1^n \cdot q^\ell \end{align}

where the first summand comes from $d=0$, which forces $n=3$ and the second summand comes from $d>0$, which forces $d=1$ and $n>0$. Continuing, this gives

\[
    \cdots = \frac{1}{2} t_0^2t_1 + \sum_{n\geq 0} \frac{1}{n!} t_1^n \cdot q^\ell = \frac{1}{2} t_0^2 t_1 + e^{t_1} q^\ell
\]

Here, $g_{jk}$ are defined by $\int_{\PP^1} T_j \cup T_k$ to give \[ g_{jk} = \left( \begin{matrix} 0 & 1 \\ 1 & 0 \end{matrix} \right) \] 

so that $g^{jk}$ is given by the inverse matrix \[ \left( \begin{matrix} 0 & 1 \\ 1 & 0 \end{matrix}\right) \]

and we define $T^i = \sum_j g^{ij}T_j$ so that $T^0 = \sum_j g^{0j} T_j = T_1$ and similarly, $T^1 = T_0$. Since $\PP^1$ is a smooth projective variety and $\Phi$ as above is the Gromov-Witten potential, the big quantum product on the cohomology $H^*(\PP^1,\C)$ is given by 

\[ T_i * T_j = \sum_k \frac{ \partial^3 \Phi}{\partial t_i \partial t_j \partial t_k} T^k\]

Here $\{T_0 = [\PP^1], T_1 = [pt]\}$ gives a basis for $H^*(X;\Q)$; we extend the above formula using linearity to give the big quantum product on $H^*(\PP^1,\C)$. The product structure is now given by

\begin{align*}
    T_0*T_0 &= T^1=T_0,\\
    T_0*T_1 &= T^0= T_1=T_1*T_0,\\
    T_1 * T_1 &= e^{t_1}q^\ell T^1 = e^{t_1}q^\ell T_0. 
\end{align*}

note that the product structure is commutative since all cohomology classes have even degree.




\subsection{Genus-0 Gromov-Witten invariants of $\PP^2$}

\label{sec: GWinvtsP2}

By the Point Mapping Axiom, we know that the only possibly non-zero GW invariants for $d=0$ are $\left<I_{0,3,0}\right>([pt],[\PP^2],[\PP^2])$ or $\left<I_{0,3,0}\right>([\ell],[\ell],[\PP^2])$. Both are

\begin{align*}
    \int_{\PP^2}[pt]\cup[\PP^2]\cup[\PP^2] &= \int_{\PP^2}[\ell]\cup[\ell]\cup[\PP^2] \\
    &= 1.
\end{align*}

For $d>0$, the Fundamental Class Axiom implies that $\left<I_{g,n,\beta}\right>(\alpha_1,\dots, \alpha_{n-1},[\PP^2])$ is $0$ for $d>0$ and $n\geq 1$. So we need only consider $I_{g,n,\beta}(\alpha_1, \dots, \alpha_n)$ of the form

\[ \left< I_{g,n,\beta}\right>([pt]^{\bullet a}, [\ell]^{\bullet b}) \] when $\beta=d\ell$ with $d>0$. The Degree Axiom again gives

\begin{align*}
\sum_{i=1}^n \deg \alpha_i &= 2(1-g)\dim X - 2\int_\beta \omega_X + 2(3g-3+n)\\
&= -2 + 6d+2n
\end{align*} 

For $d=1$, this forces $n\geq 2$. Now $n=2$ gives $\sum_{i=1}^n \deg \alpha_i = 8$ leaves only $\left<I_{0,2,1}\right>([pt],[pt])$. For $n=3$, $\sum_{i=1}^n \deg\alpha_i=10$ leaves only

\begin{align*}
    \left<I_{0,3,1}\right>([pt],[pt],[\ell]) &= \left( \int_\ell [\ell]\right) \cdot \left< I_{0,2,1}\right>\left([pt],[pt]\right)\\
    &= \left<I_{0,2,1}\right>([pt],[pt])
\end{align*} 

Similarly, $n=4$ leaves only 

\begin{align*}
    \left<I_{0,4,1}\right>([pt],[pt],[\ell],[\ell]) &= \int_\ell[\ell] \cdot \left<I_{0,3,1}\right>([pt],[pt],[\ell])\\
    &= \left( \int_\ell [\ell] \right)^2 \left< I_{0,2,1}\right> ([pt],[pt])\\
    &= \left< I_{0,2,1} \right>([pt],[pt]).
\end{align*}

Since we are inserting only classes of the form $([pt]^{\bullet a},[\ell]^{\bullet b})$, we know that $\sum_{i=1}^n\deg \alpha_i = 4a + 2b$. This gives

\begin{align*}
    \sum_{i=1}^n \deg\alpha_i &= 4a + 2b\\
    &= -2+6d+2n\\
    &= -2 + 6d + 2(a+b) &\text{ gives }\\
    2a &= -2 + 6d\\
    a&= -1+ 3d
\end{align*}


When $d=1$ as above, this leaves only 

\begin{align*}
     \left< I_{0,n+2,1}\right>([pt],[pt], [\ell]^{\bullet n}) &= \left( \int_\ell [\ell] \right)^n \cdot \left< I_{0,2,1} \right> ([pt],[pt]) \\
     &= \left< I_{0,2,1}\right>([pt],[pt]).\end{align*}

When $d=2$ as above, this leaves only $a=5$ and

\begin{align*}
    \left<I_{0,n+5,2}\right>([pt]^{\bullet 5},[\ell]^{\bullet n}) &= \left( \int_{2\ell} [\ell]\right)^n \cdot \left<I_{0,5,2}\right>([pt]^{\bullet 5} )
\end{align*} is determined by the number of conics through 5 points in general position \cite{coxkatxMSAG}. Similarly, $d>0$ reduces to \[ N_d := \left<I_{0,3d-1,d}\right>([pt]^{\bullet 3d-1} ). \]


\subsection{Computing $QH^*(\PP^2)$}
\label{sec: 4.5}
Fix the basis $T_0 = [\PP^2], T_1=[\ell], T_2=[pt]$ for $H^*(\PP^2, \Q)$ and introduce formal parameters $t_0, t_1, t_2$ with degrees $\deg t_0 = -2, \deg t_1 = 0, \deg t_2 = 2$. Similar to the case for $\PP^1$, let $\gamma = t_0T_0 + t_1T_1 + t_2T_2$ and write

\[ \Phi(\gamma) = \sum_{n=0}^\infty \sum_{\beta\in H_2(\PP^2,\Z)} \frac{1}{n!} \left<I_{0,n,\beta}\right>(\gamma^n) q^\beta \] as a formal power series in $t_0, t_1, t_2$. Now this equals

\begin{align*}
\cdots &= \frac{1}{2}( t_0^2 t_2 + t_0 t_1^2 ) + \sum_{d=1}^\infty \sum_{n\geq 0} \frac{(dt_1)^n}{n!} N_d \cdot \frac{t_2^{3d-1}}{(3d-1)!} \cdot q^{d\ell}\\
&= \frac{1}{2}( t_0^2 t_2 + t_0 t_1^2 ) + \sum_{d=1}^\infty e^{dt_1} N_d \cdot \frac{ t_2^{3d-1}}{(3d-1)!}q^{d\ell} \end{align*} 

where the two summands in the first line come from $d=0$ and $d>0$ from considerations of the previous Section~\ref{sec: GWinvtsP2}. In order to solve for $N_d$, here we make clever use of the Witten-Dijkgraaf-Verlinde-Verlinde equation \[ \sum_{a,b} \frac{ \partial^3 \Phi}{\partial t_i \partial t_j \partial t_a} g^{ab} \frac{ \partial^3 \Phi }{\partial t_b \partial t_k \partial t_\ell} = (-1)^{\deg t_i (\deg t_j + \deg t_k)} \sum_{a,b} \frac{ \partial^3 \Phi}{\partial t_j \partial t_k \partial t_a} g^{ab} \frac{ \partial^3 \Phi}{\partial t_b \partial t_i \partial t_\ell} \] for all $i,j,k,\ell$. Again, for $\PP^2$, all cohomology classes have even degree, so we can ignore the factor of $(-1)$ on the right-hand side. Following \cite{coxkatxMSAG}, we set $(i,j,k,\ell) = (1,1,2,2)$ which gives

\begin{align*}
    \Phi_{011}\Phi_{222} + \Phi_{111}\Phi_{221} + \Phi_{211}\Phi_{220} &= \Phi_{021}\Phi_{212} + \Phi_{121}\Phi_{211} + \Phi_{221}\Phi_{210}
\end{align*}

Now $\Phi_{0jk} = g_{jk}$ and the partial derivative operators pairwise commute $(i.e., \Phi_{ij}=\Phi_{ji}$ for all $i,j$) since all cohomology classes have even degree, so that this simplifies to 

\[ \Phi_{222} + \Phi_{111}\Phi_{122} = \Phi_{211}^2 \] 

When we write

\begin{align*}
    \Phi_{222} &= \sum_{\ell=1}^\infty N_{\ell+1} e^{(\ell + 1)t_1} \cdot \frac{ t_2^{3\ell - 1}}{(3\ell-1)!} q^{\ell+1}\\
    &= \sum_{\ell=1}^\infty N_{\ell+1} \sum_{k=0}^\infty \frac{ t_1^{(\ell+1)k}}{k!} \cdot \frac{ t_2^{3\ell -1 }}{(3\ell-1)!} q^{\ell+1}, 
\end{align*}

the only terms without $t_1$ are from $k=0$, as in

\[ \sum_{\ell=1}^\infty N_{\ell+1} \frac{ t_2^{3\ell -1 }}{(3\ell-1)!} \cdot q^{\ell + 1 } \] 

Similarly, the only terms in $\Phi_{111}$ without $t_1$ are

\[ \sum_{d>0} d^3 N_d \frac{ t_2^{3d-1}}{(3d-1)!} q^d \] 

and in $\Phi_{122}$ without $t_1$ are \[ \sum_{d>0}d N_d \frac{ t_2^{3d-3}}{(3d-3)!} q^d \] 

Comparing terms on the left- and right-hand sides without $t_1$ gives

\begin{align*}
    \left[ \sum_{\ell=1}^\infty N_{\ell+1} \frac{ t_2^{3\ell -1 }}{(3\ell-1)!} \cdot q^{\ell + 1 } \right] + \left[\sum_{d_1>0} d_1^3 N_{d_1} \frac{ t_2^{3d_1-1}}{(3d_1-1)!} q^{d_1} \right] \cdot \left[ \sum_{d_2>0}d_2 N_{d_2} \frac{ t_2^{3d_2-3}}{(3d_2-3)!} q^{d_2}    \right] =  \cdots \\
  \left[   \sum_{ \makecell{ d_1 + d_2 = d,\\ d_1,d_2>0}} d_1^2 d_2^2 N_{d_1} N_{d_2} \frac{ t_2^{3d-4} }{(3d_1-2)! (3d_2-2)!} q^d  \right] 
\end{align*}

gives

\begin{align*}
    \left[ \sum_{d=2}^\infty N_d \frac{ t_2^{3d-4}}{(3d-4)!} q^d \right]  +  \left[ \sum_{ \makecell{ d_1 + d_2 = d, \\ d_1,d_2>0}} d_1^3 d_2 N_{d_1} N_{d_2} \frac{t_2^{3d - 4}}{(3d_1-1)!(3d_2-3)!} q^d \right] = \cdots  \\
     \left[ \sum_{ \makecell{ d_1 + d_2 = d,\\ d_1,d_2>0}} d_1^2 d_2^2 N_{d_1} N_{d_2} \frac{ t_2^{3d-4} }{(3d_1-2)! (3d_2-2)!} q^d \right] 
\end{align*}

so that combining coefficients on $t_2^{3d-4}q^d$ gives 

\begin{align*}
    \frac{N_d}{(3d-4)!} &= \sum_{ \makecell{ d_1+d_2=d, \\ d_1,d_2>0}} \frac{d_1^2d_2^2 N_{d_1}N_{d_2}}{(3d_1-2)!(3d_2-2)!} - \frac{d_1^3 d_2 N_{d_1} N_{d_2}}{(3d_1-1)!(3d_2-3)!} 
\end{align*}

and

\begin{align*}
N_d &= \sum_{ \makecell{ d_1 + d_2 = d, \\ d_1, d_2>0} } d_1^2 d_2^2 N_{d_1} N_{d_2} \binom{3d-4}{3d_1-2} - \binom{3d-4}{3d_1-1} d_1^3 d_2 N_{d_1} N_{d_2} 
\end{align*}

recovers the recursive formula due to Kontsevich-Manin \cite{KontsevichManin_1994} for $N_d$ on $\PP^2$. To use this formula, we start with the input data that there is a unique line through 2 distinct points of $\PP^2: N_1=1$. From here, we see that 

\begin{align*}
    N_2 &= 1\cdot \binom{2}{1} - \binom{2}{2}\cdot 1 \\
    &= 1 \end{align*}

says that there is 1 unique conic through 5 points in general position,

    \begin{align*}
    N_3 &= 4 \cdot N_1 N_2 \binom{5}{1} - \binom{5}{2}\cdot 2 + 4\cdot \binom{5}{4} - \binom{5}{5} \cdot 8\\
    &= 12
\end{align*}

says that there are 12 rational cubic curves through 8 points in general position, and so on.


\section{Genus $0$ Gromov-Witten potential of a genus $g>0$ Riemann surface}

For $\Sigma_g$ a (smooth) genus $g$ Riemann surface, we fix a basis of $H^*(\Sigma_g, \Z)$ $\{ T_0= [\Sigma_g], T_1 = [a_1], T_2 =[b_1], \dots, T_{2g-1} = [a_g], T_{2g} = [b_g], T_{2g+1} = [pt]\}$ such that 

\[ \begin{cases} \int_{\Sigma_g} [a_i] \cup [b_j] = \delta_i^j\\ \int_{\Sigma_g} a_i \cup a_j = 0 = \int_{\Sigma_g} b_i\cup b_j  \end{cases}  \] for all $1 \leq i,j \leq g$. Here, $\{[a_1],[b_1],\dots [a_g],[b_g]\}$ are a basis of $H^1(\Sigma_g, \Z)$ with the above properties. Since the only morphisms from $\PP^1 \rightarrow \Sigma_g$ are constant, we can assume that $\beta = 0 \in H_2(\Sigma_g, \Z)$. In order to be stable, this implies that $n\geq 3$. By the Point Mapping Axiom, the only possibly non-zero genus 0 Gromov-Witten invariants of $\Sigma_g$ (up to permuting the entries) are 

\begin{align*}
\left<I_{0,3,0}\right>\left( [pt],[\Sigma_g],[\Sigma_g] \right) &= 1\\
\left<I_{0,3,0}\right> \left( [\Sigma_g], [a_i],[b_i] \right) &= 1
\end{align*} 

for all $1\leq i \leq g$. This gives, for $\gamma = \sum_{i=0}^{2g+1}t_iT_i$ with formal variables $t_i$ such that $\deg t_i = \deg T_i-2$ as above,

\begin{align*}
    \Phi(\gamma) &= \frac{1}{2}t_0^2 t_2 - \sum_{j=1}^g t_0 t_{2j-1} t_{2j}
\end{align*}

where the minus sign is due to $\epsilon(\alpha)$ in the formula

\begin{align*}
    \frac{1}{n!} \left< I_{0,n,\beta} \right> (\gamma^n) &= \sum_{|\alpha|=n} \epsilon(\alpha) \left<I_{0,n,\beta} \right> (T^\alpha) \frac{t^\alpha}{\alpha!} \end{align*}

where $\epsilon(\alpha)=\pm 1$ is determined by 

\begin{align*}
        (t_0T_0)^{a_0}(t_1T_1)^{a_1} \cdots (t_mT_m)^{a_m} &= \epsilon(\alpha) T_0^{a_0}\cdots T_m^{a_m}t_0^{a_0} \cdots t_m^{a_m} 
    \end{align*}

This also implies that the big quantum cohomology ring is isomorphic to the small quantum cohomology ring, as $g_{ij}$ is given by the matrix (where labels on rows and columns are included)

\[ g_{ij} = \begin{array}{c|c|c|c|c|c|c| c}
& \left[\Sigma_g\right]& \left[a_1\right] & \left[b_1\right] & \cdots & \left[a_g\right] & \left[b_g\right] & \left[pt\right] \\
\hline
\left[\Sigma_g\right] & 0 & 0 & 0 & \cdots & 0 & 0 & 1 \\
\hline 
\left[a_1\right] & 0 & 0 & 1 & \cdots & 0 & 0 & 0 \\
\hline 
\left[b_1\right] & 0 & 1 & 0 & \cdots & 0 & 0 & 0 \\
\hline 
\vdots & \vdots & \vdots & \vdots & \ddots & \vdots & \vdots & \vdots \\
\hline 
\left[a_g\right] & 0 & 0 & 0 & \cdots & 0 & 1 & 0 \\
\hline 
\left[b_g\right] & 0 & 0 & 0 & \cdots & 1 & 0 & 0  \\
\hline 
\left[pt\right] & 1 & 0 & 0 & \cdots & 0 & 0 & 0 \\
 \end{array} = g^{ij} \]

 so that $T^0 = [pt], T^{2j-1}=T_{2j}$ and $T^{2j}=T_{2j-1}$ for $1\leq j \leq g$, and $T^{2g+1}=T_0$. Now the big (and small) quantum product agree with the usual cup product, via

 \begin{align*}
     T_0 * T_0 &= T^{2g+1}=T_0\\
     T_0*T_i &= T_i \text{ for } T_i \in \{[a_1],[b_1], \dots, [a_g],[b_g] \}\\
     T_0 * T_{2g+1} &= T_{2g+1}\\
     T_i^2 &= 0 \text{ for } 1\leq i \leq 2g\\
     [a_i]*[b_i] &= T_{2g+1} = -[b_i]*[a_i]\\
     T_i*T_{2g+1}&= 0 \text{ for all }i.      
 \end{align*}

Here we make use of the fact that the $t_i$ supercommute, as in \[ t_i t_j = (-1)^{\deg t_i \deg t_j} t_j t_i \] and $t_i^2=0$ if $t_i$ has odd degree. We define the partial derivative operator $\frac{\partial}{\partial t_i}$ by \[ \frac{\partial }{\partial t_i} (t_i^k \cdot t^\alpha) = kt_i^{k-1} \cdot t^\alpha \] for any monomial $t^\alpha \in \C[t_0, \dots, t_m]$ not involving $t_i$. It follows that \[ \frac{ \partial^2 F}{\partial t_i \partial t_j} = (-1)^{\deg t_i \deg t_j } \frac{ \partial^2 F}{\partial t_j \partial t_i}. \] 

\section{Quiver-Theoretic Approach for $J$-functions of Quiver Flag Varieties}

In this section, we will describe Kalashnikov's \cite{Kalashnikov} quiver-theoretic method for extracting Givental's cohomological small $J$-function for a Fano quiver flag zero locus given by the quiver $Q$ and representation-theoretic bundle $E_G$ with $j: X \rightarrow M_Q$ the embedding of $X$ into the ambient quiver flag variety given by a generic section of $E_G$. We give $\C \PP^1$ as an example. We then compare to the conventions in \cite{coxkatxMSAG} and recover the genus $0$ Gromov-Witten potential for $\C \PP^1$ from $J_{\PP^1}(0,z)$ in Kalashnikov's convention. In order to define Givental's small $J$-function, we'll require the following definition \cite[302]{coxkatxMSAG}. 

\begin{definition} Let $X$ be a smooth projective variety over $\C$. Given $\gamma_1, \dots, \gamma_n \in H^*(X, \Z)$ and non-negative integers $d_i$ for all $i$, the $n$-point gravitational correlators 

\[ \left< \tau_{d_1} \gamma_1, \tau_{d_2}\gamma_2, \dots, \tau_{d_n}\gamma_n \right>_{g,\beta} := \int_{ \left[ \overline{M}_{g,n}(X,\beta)\right]^{\text{virt}}} \prod_{i=1}^n (c_1 (\mathcal{L}_i)^{d_i} \cup e_i^*(\gamma_i) ) \]

\end{definition}

The $n$-point gravitational correlators satisfy, in particular, the \textbf{Dilaton Axiom}\label{axiom: dilatonAxiom} \[ \left< \tau_1, \tau_{d_1}\gamma_1, \dots, \tau_{d_n}\gamma_n \right> = (2g-2+n) \cdot \left< \tau_{d_1}\gamma_1, \dots, \tau_{d_n}\gamma_n \right>. \] 
Here $\mathcal{L}_i$ is the ``cotangent line at the $i$th marked point," i.e. the line bundle on $\overline{\mathcal{M}}_{g,n}(X,\beta)$ whose fiber over the stable map $(f: C \rightarrow X, p_1, \dots, p_n )$ is the cotangent space $T^*_{p_i}C$. \\

In order to describe $\PP^1$ as a toric quiver variety, we'll need the following definitions.

\begin{definition}\begin{itemize} \text{ } \\
    \item[(i)] \cite{kirillov2016quiver}  A \textbf{quiver} $Q = (Q_0, Q_1)$ is a collection of vertices $\{v_0, \dots, v_\rho\} = Q_0$ and edges $\{a_1, \dots, a_N \} =Q_1$ with head and tail maps 
 \[ h, t: Q_1 \rightarrow Q_0 \]

for all $a \in Q_1$. \\

\item[(ii)] The \textbf{path algebra} $\Bbbk Q$ consists of $\Bbbk-$linear combinations of composable paths in $Q$. The path algebra is $\Z$-graded, with $(\Bbbk Q)_0 = Q_0$ and $(\Bbbk Q)_1 = Q_1$. \\
\item[(iii)] A \textbf{representation of $Q$} is a choice of finite-dimensional vector space $V_i$ of finite-dimensional $\Bbbk$-vector space for each vertex, and an element $x_a\in \mathrm{Hom}_{\Bbbk-v.s.}(V_{t(a)},V_{h(a)})$ for all $a\in Q_1$ which can be composed in agreement with composable paths in $Q$. That is, \[ \text{Rep}(Q) \stackrel{\simeq}{\leftrightarrow} \Bbbk Q\text{-Mod}.\]

We emphasize that we simultaneously fix a basis for each $V_i$ as $i$ ranges over $Q_0$. 

\item[(iv)] The \textbf{dimension vector} of $M \in \text{Rep}(Q)$ is $\mathbb{r} = (r_0, \dots, r_\rho) \in \N^{\rho + 1}$.  \\
\end{itemize}
\end{definition}

Following \cite{CrawDuke2011}, we impose $r_0=1$ and require that there are no arrows to vertex $v_0$. Fix $(r_1, \dots, r_\rho) \in \N^\rho$. We can view

\[ \mathrm{Rep}(Q,\mathbb{r}) = \bigoplus_{a\in Q_1} \mathrm{Hom}(\Bbbk^{r_{t(a)}}, \Bbbk^{r_{h(a)}}).\]

There is a natural action of $\prod_{i=0}^\rho \mathrm{GL}(d_i)$ on $M\in \mathrm{Rep}(Q)$ which acts on edges in the following manner. For $w=(w_a) \in \mathrm{Rep}(Q,\mathbb{r})$ and $g=(g_0, \dots, g_\rho) \in \prod_{i=0}^\rho \mathrm{GL}(d_i)$, \[ (g\cdot w)_a = g_{h(a)}\circ w_a \circ g_{t(a)}^{-1}. \] Now $\mathbb{G}_m$ acts trivially via diagonal matrices, but by setting the first entry equal to 1, we see that \[ G = \prod_{i=1}^\rho \mathrm{GL}(d_i) \] acts effectively. 

\begin{definition}
\begin{itemize}
    \item[(v)] We use the stability condition $\theta \in \mathrm{Hom}(G, \Bbbk^*)$ given by $\theta(g) = \prod_{i=1}^\rho \det(g_i)$. 
\end{itemize}
\end{definition}

For this choice of stability condition, all $\theta$-semistable points are $\theta$-stable. Let $s_i = \sum_{a\in Q_1, h(a)=v_i} r_{t(a)}$ and view \[ \mathrm{Rep}(Q,\mathbb{r}) = \bigoplus_{i=1}^\rho \mathrm{Hom}_{\Bbbk-v.s.}(\Bbbk^{s_i}, \Bbbk^{r_i}). \] Then $w = (w_i)_{i=1}^\rho$ is $\theta$ stable if and only if each $w_i$ is surjective. 

\begin{definition}The \textbf{GIT quotient of the moduli space of $\theta$-semistable representations of $Q$ with dimension vector $\mathbb{r}$, by $G$} is \[ \mathcal{M}^{ss} \sslash_\theta G := \faktor{\mathcal{M}^{ss}(\theta,\mathbb{r})}{G}. \] 
\end{definition}

Here, we get an abelian quotient if and only if the dimension vector $\mathbb{r}=(1,1,\dots, 1)$, which \cite{AltmannHille} call ``thin sincere" representations. Otherwise, the non-commutativity of $GL(m)$ for $m>1$ gives a non-commutative quotient. To introduce more notation:

\[
\begin{tikzcd}
\circled{1} \arrow[r,"n"] & \circled{k}
\end{tikzcd}\] 

denotes the quiver 
\[\begin{tikzcd}
	{v_0 \text{ } \bullet} && {\bullet \text{ }v_1}
	\arrow["{a_1}", curve={height=-12pt}, from=1-1, to=1-3]
	\arrow["{a_2}", from=1-1, to=1-3]
	\arrow["\begin{array}{c} \vdots \\ a_n \end{array}"{marking, allow upside down}, curve={height=18pt}, from=1-1, to=1-3]
\end{tikzcd}\]

with dimension vector $\mathbb{r}=(1,k)$. This gives the Grassmanian \[ \faktor{ \mathcal{M}(\theta, \mathbb{r}=(1,k)) }{G} = \text{ }\faktor{ \{ \text{ surjections from }\Bbbk^n \rightarrow \Bbbk^k \}}{\mathrm{GL}(k) } \cong \mathrm{Gr}_\Bbbk(k,n).\]

In particular, 

\[
\begin{tikzcd}
\circled{1} \arrow[r,"n"] & \circled{1}
\end{tikzcd}\]

gives $\PP^{n-1}$, and 

\[ \begin{tikzcd}
\circled{1} \arrow[r,"n"] & \circled{k} \arrow[r,"1"] & \circled{a} \arrow[r,"1"] & \circled{b}
\end{tikzcd}\]

gives the partial flag variety $\mathcal{F}l(n,k,a,b)$ with $n>k>a>b$. Our running example will be

\[
\begin{tikzcd}
\circled{1} \arrow[r,"2"] & \circled{1}
\end{tikzcd}\]

for $\PP^1$. \cite{Kalashnikov} gives a formula for Givental's small $J$-function for zero loci in quiver flag varieties. To do this, one first constructs the $I$-function on the mirror, and then invokes mirror symmetry by equating a generating function for the $I$ function on the mirror with the small $J$-function on the original space. For zero loci in quiver flag varieties, Kalashnikov first gives the formula 

\begin{align*}
    I_{T_{M_Q}}(\tilde{d}) &= \frac{ \prod_{a\in Q_1^{ab}} \prod_{m\leq 0} (D_a + mz) }{ \prod_{a\in Q_1^{ab}} \prod_{m\leq \left< \tilde{d}, D_a \right>} (D_a + mz) } \cdot \frac{ \prod_{i=1}^\rho \prod_{j\neq k} \prod_{m\leq \left< \tilde{d}, D_{ij}-D_{ik} \right>} (D_{ij} - D_{ik} + mz ) }{ \prod_{i=1}^\rho \prod_{j\neq k} \prod_{m\leq 0} (D_{ij} - D_{ik} + mz ) }
    \end{align*}

for $\tilde{d} \in NE_1(M_{Q^{ab}})$.  When $\mathbb{r}=(1,1,\dots, 1)$, $Q = Q^{ab}$, which also implies that all vertices have second index $1$. Here $D_{ij}$ is the divisor corresponding to the tautological bundle $W_{ij}$ for vertex $ij$, and $D_a := -D_{s(a)} + D_{t(a)}$ is the divisor on $M_{Q^{ab}}$ corresponding to the arrow $a \in Q_1^{ab}$. Also, for $\PP^n, NE_1(M_{Q^{ab}}) \cong \N$. The $I$ function of the zero locus $X \subset M_Q$ is then assembled into the generating function

\[ I_{X, M_Q}(z) = \sum_{d \in NE_1(M_Q)} \sum_{\tilde{d}\mapsto d} (-1)^{\epsilon(d)} \cdot q^d \cdot I_{T_X}(\tilde{d}). \] 

For $\PP^1$, all $\epsilon(d)=1$. For $X$ a Fano quiver flag zero locus given by $(Q,E_G)$ with $j: X\rightarrow M_Q$ the embedding of $X$ into the ambient quiver flag variety, the Givental's cohomological small $J$-function is then defined by the change of coordinates

\[ J_X(0,z) := e^{-c/z} \cdot j^* I_{X,M_Q}(z). \]

For $\PP^1$, Kalashnikov's formula gives

\[ I_{T_{M_Q}}(\tilde{d}) = \frac{ \left[ \prod_{m\leq 0} ({h}+mz ) \right]^2 }{ \left[ \prod_{m\leq \left< \tilde{d}, D_a \right>}( h + mz) \right]^2 }\]

for the $I$-function on the mirror, which we assemble into the generating function 

\[ I_{\PP^1, \mathcal{M}_Q}(z) = \sum_{d\geq 0} q^d \frac{1}{ \left[ \prod_{m=1}^d (h+mz)\right]^2 }.\]

The change of coordinates to the mirror then gives

\begin{equation} \label{eqn: KalashJP1} J_{\PP^1}(0,z) = e^{-c/z} \cdot \sum_{d\geq 0} q^d \cdot \frac{ 1}{ \left[ \prod_{m=1}^d (h+mz) \right]^2 }. \end{equation} 

In the conventions of \cite{coxkatxMSAG}: if $\{T_0, \dots, T_\ell\}$ is an ordered basis for $H^*(\PP^1,\Z)$ for $X$ a smooth projective variety over $\C$ with $\{t_i\}_{i=0}^\ell$ formal variables such that $|t_i| = |T_i|-2$, then

\begin{equation} \label{eqn: coxKatzIP1} I_{\PP^1} = e^{(t_0 + t_1H )/h} \cdot \sum_{d\geq 0} \frac{ q^d}{ \left[ (H+h)(H+2h)\cdots (H+dh)\right]^2}  \end{equation}

where $d=0$ gives the summand $1$ in $J_{\PP^1}$.

\begin{definition} \label{defn: coxKatzJP1} \cite[357]{coxkatxMSAG} In the conventions of Cox-Katz, we'll use as definition for the small $J$-function for $\PP^1$ the following expression in terms of 2-point gravitational correlators: 

\[ J_{\PP^1} = e^{(t_0 + t_1H)/h} \cdot \sum_{d\geq 0} e^{dt_1} \left( \frac{ \left< \tau_{2d-1}H,1\right>_{0,d} }{h^{2d}} + \frac{ \left< \tau_{2d},1\right>_{0,d} H }{h^{2d+1} } \right) \]

where $d=0$ contributes the summand $1$ above. 

\end{definition}

From Definition~\ref{defn: coxKatzJP1}, one might ask how to recover the gravitational correlators $\left<\tau_{2d},1\right>$ from Kalashnikov's formula for $J_{\PP^1}(0,z)$. To do this, we substitute to translate between conventions in \cite{Kalashnikov} and \cite{coxhgscoordinatering}, then view $I=J$ as an instance of Givental's mirror theorem for $\PP^1$. Finally, we expand in powers of $H$ while using the fact that $H^2=0$, for $H$ the hyperplane class in the conventions of \cite{coxkatxMSAG}. That is, we first use the substitution

\begin{align*} \begin{cases} 
    z &\mapsto h, \\
    -c &\mapsto t_0 + t_1H, \\
    q^d &\mapsto e^{dt_1},\\
    h &\mapsto H. \end{cases} 
\end{align*}

This substitution now gives an equality between $J_{\PP^1}(0,z)$ from \cite{Kalashnikov} in Equation~\ref{eqn: KalashJP1} and $I_{\PP^1}$ in \cite{coxkatxMSAG} from Equation~\ref{eqn: coxKatzIP1}.\\

Second, expanding $I_{\PP^1}$ from \cite{coxkatxMSAG} in Equation~\ref{eqn: coxKatzIP1} in $H$, and using the fact that $H^2=0$ yields the following. For $d=1$ in Equation~\ref{eqn: coxKatzIP1}, we have 

\begin{align*}
\frac{e^{t_1}}{(H+h)^2} &= \frac{ e^{t_1}}{h^2 + 2hH}\\
&= \frac{ e^{t_1}}{h^2(1 - \faktor{-2H}{h})} \\
&= \frac{e^{t_1}}{h^2} \left[ 1 - \faktor{2H}{h} \right] \end{align*}

gives that \[ \left< \tau_1 H, 1\right>=1 \text{ and }\left< \tau_2, 1\right> = -2. \] Next, $d=2$ in the same Equation~\ref{eqn: coxKatzIP1} gives 

\begin{align*}
\frac{e^{2t_1}}{\left[(H+h)(H+2h)\right]^2} &= \frac{q^2}{ 4h^4 + 2Hh^3(4+2)}\\
&= \frac{q^2}{4h^4 + 12Hh^3}\\
&= \frac{q^2}{4h^4 \cdot (1 - \faktor{-3H}{h} ) }   
\end{align*} 

so that  \[ \left<\tau_3 H,1\right> = \frac{1}{4} \text{  and } \left< \tau_4, 1\right> = \frac{-3}{4}.\] At this point, the general formula becomes clear: For a given $d>1$, we'll have

\begin{align*}
\frac{q^d}{(d!)^2\cdot h^{2d} + 2H\cdot h^{2d-1}\left( \sum_{j=1}^d \frac{d!}{j}(d!) \right)} &= \frac{ q^d }{(d!)^2 \cdot h^{2d} \left( 1 - -2\cdot \sum_{j=1}^d \frac{1}{j} \cdot \frac{H}{h} \right) }\\
&= \frac{q^d}{(d!)^2 \cdot h^{2d} } \cdot \left( 1 - 2\sum_{j=1}^d \frac{1}{j} \cdot \frac{H}{h} \right) \end{align*} 

gives

\begin{itemize}
    \item[(i)] \[ \left< \tau_{2d-1}H,1\right>_{0,d} = \frac{ 1}{(d!)^2} \] and
    \item[(ii)] \cite{Pandharipande1999TheTE} 
    
    \[ \left< \tau_{2d}, 1\right>_{0,d} = \frac{-2}{(d!)^2} \cdot \left( 1 + \frac{1}{2} + \cdots + \frac{1}{d} \right). \]
\end{itemize}

From the \hyperref[axiom: dilatonAxiom]{Dilaton Axiom},

\begin{align*}
    \left<\tau_1, H\right>_{0,1} = (-2+1) \cdot \left< H \right>_{0,1} &= - \int_{ \left[M_{0,1}(\PP^1, 1) \right]^{\text{virt}}} [pt] \\
    &= - \int_{\PP^1} [pt] &\text{ from } M_{0,1}(\PP^1,1) \cong \PP^1, \\
    &&\text{ which we can also write as}\\
    &= -\int_{\PP^1} [pt] \cup [pt] \cup [pt]\\
    &= - \left< I_{0,3,0}\right>([pt], [\PP^1], [\PP^1])\\
    &= -1.
\end{align*}

We also have the \textbf{Fundamental Class Axiom}\label{axiom: FundClassAxiom}:

\[ \left< \tau_{d_1}\gamma_1, \dots, \tau_{d_{n-1}}\gamma_{n-1}, 1 \right>_{g,\beta} = \sum_{i=1}^{n-1} \left< \tau_{d_1}\gamma_1, \dots, \tau_{d_i-1}\gamma_{i-1}, \tau_{d_i - 1} \gamma_i, \dots, \tau_{d_{n-1}}\gamma_{n-1} \right>\] where any term involving $\tau_{-1}$ is taken to be zero. This gives 

\begin{align*}
    \left< \tau_1 H, 1 \right>_{0,1} &= \left< H \right>_{0,1} &\text{ by the Fundamental Class Axiom}\\
    &= \left< H,H,H\right>_{0,1} \\
    &= \left< I_{0,3,1} \right>(H,H,H) \\
    &= \left< I_{0,3,1} \right>([pt],[pt],[pt])\\
    &= 1 
\end{align*}

from Subsection~\ref{subsec: GWInvtsP1}. Both the Dilaton and Fundamental Class Axiom are valid if $n+2g\geq 4$, or $\beta\neq 0$ and $n\geq 1$. Note that $\overline{\mathcal{M}}_{0,1}(\PP^1,1) \cong \PP^1$, where the isomorphism takes the map $f: (C,p) \rightarrow \PP^1$ to $f(p)$, which identifies \[ \mathcal{L}_1 \cong \mathcal{T}^*\PP^1 \cong \mathcal{O}_{\PP^1}(-2). \] 

At this point, we now have obtained a list of all non-zero genus 0 Gromov-Witten invariants of $\PP^1$ from Kalashnikov's $J_{\PP^1}(0,z)$:

\begin{align*}
    \begin{cases} \left< I_{0,3,0}\right>([pt],[\PP^1],[\PP^1]) = \int_{\PP^1} [pt]\cup[\PP^1]\cup[\PP^1] &= 1, \text{ while others reduce to}\\
     \left<I_{0,n,1}\right>([pt]^{\bullet n}) &= 1. \end{cases}
\end{align*}

From this list, we recover the genus $0$ Gromov-Witten potential for $\PP^1$, with $\gamma = \sum_{i=0}^2 t_i T_i$ from \eqref{eqn: GWPotP1}:

\begin{align*}
    \Phi(\gamma) &= \sum_{n\geq 0}\sum_{\beta \in H_2(\PP^1, \Z) } \frac{1}{n!} \left< I_{g,n,\beta} \right> (\gamma^n) \cdot q^\beta \\
    &= \frac{1}{2}t_0^2 t_1 + \sum_{n\geq 0} \frac{t_1^n}{n!} \\
    &= \frac{1}{2} t_0^2 t_1 + e^{t_1} \cdot q^\ell. 
\end{align*}

\section{Tropical approach}

\subsection{Counting lattice paths in lattice polygons} The degree $d$ genus $0$ Gromov-Witten invariants for $\mathbb{P}^2$ can also be computed by enumerating certain lattice paths, up to multiplicity, of the Newton polygon $\Delta_d=Conv\{(0,0), (d, 0), (0,d)\}$. These lattice paths compute $N_{trop}(g, \Delta)$, the number of, not necessarily irreducible, tropical curves of degree $d$, genus $g$ passing through $3d+g-1$ points in tropical general position in $\mathbb{R}^2$. Mikhalkin, in \cite{MikEnumTropGeom}, proves that each such curve is the Hausdorff limit of the amoeba $Log(V_{k_{\alpha}})$ of a subsequence of $J_t-$holomorphic curves $V_{k_{\alpha}}$ of the same degree and genus passing through $3d+g-1$ points in general position in $(\mathbb{C}^*)^2$. 

A precursor to this result is the paper \cite{MikCountingCurvesLatticePaths}, in which Mikhalkin describes the algorithm for counting the desired lattice paths. Fix a linear function $\lambda: \mathbb{R}^2\rightarrow \mathbb{R}$ with irrational slope (equivalently, such that $\lambda$ restricted to the lattice $\mathbb{Z}^2$ is injective). Let $d>0$ be an integer and let $n=\dfrac{d(d+3)}{2}-1$. A lattice path is a path $\gamma: [0,n]\rightarrow \mathbb{R}^2$ such that $j\in 0,\dots, n$, $\gamma(j)\in \mathbb{Z}^2$, and $\gamma$ restricted to $[j-1,j]$ is affine-linear. The desired lattice paths in $\Delta_d$ are those that increase with respect to $\lambda$. For example, if $\epsilon$ is small and irrational, and $\lambda=x-\epsilon y$, then the path $\gamma: [0,8]\rightarrow \mathbb{R}^2$ in $\Delta_3$ in Figure \ref{lattice_p1} with $\gamma(0)=(0,3)$ and $\gamma(8)=(3,0)$ is $\lambda$-increasing (one can directly compute that $\lambda(\gamma(0))=-3\epsilon, \lambda(\gamma(1))=-2\epsilon, \dots, \lambda(\gamma(7))=2, \lambda(\gamma(8))=3$).

\FloatBarrier
\begin{figure}[h]
\centering
\begin{tikzpicture}[
    x=0.72cm,
    y=0.72cm,
    line cap=round,
    line join=round,
    tri/.style={black,line width=0.75pt},
    latticePurple/.style={draw=purple!85!magenta,line width=2.0pt},
    plusGreen/.style={draw=green!60!black,line width=2.0pt},
    minusCyan/.style={draw=cyan!75!black,line width=2.0pt},
    ptPurple/.style={circle,fill=purple!85!magenta,inner sep=0pt,minimum size=5.2pt},
    ptGreen/.style={circle,fill=green!60!black,inner sep=0pt,minimum size=5.2pt},
    ptCyan/.style={circle,fill=cyan!75!black,inner sep=0pt,minimum size=5.2pt},
    every node/.style={font=\large}
]

\begin{scope}[shift={(0,0)}]
    \draw[tri] (0,0) -- (0,3) -- (3,0) -- cycle;

    \draw[latticePurple]
        (0,3) -- (0,2) -- (0,1) -- (1,2) -- (1,1) -- (1,0)
        -- (2,1) -- (2,0) -- (3,0);

    \foreach \P in {(0,3),(0,2),(0,1),(1,2),(1,1),(1,0),(2,1),(2,0),(3,0)}
        \node[ptPurple] at \P {};

    \node[anchor=south east] at (-0.15,3.15) {$p$};
    \node[anchor=north west] at (3.12,0.05) {$q$};
\end{scope}

\begin{scope}[shift={(4.85,0)}]
    \draw[tri] (0,0) -- (0,3) -- (3,0) -- cycle;

    \fill[green!35,opacity=0.65] (0,3) -- (1,2) -- (0,1) -- (0,2) -- cycle;
    \fill[green!35,opacity=0.65] (1,2) -- (2,1) -- (1,0) -- (1,1) -- cycle;
    \fill[green!35,opacity=0.65] (2,1) -- (3,0) -- (2,0) -- cycle;

    \draw[latticePurple]
        (0,3) -- (0,2) -- (0,1) -- (1,2) -- (1,1) -- (1,0)
        -- (2,1) -- (2,0) -- (3,0);

    \draw[plusGreen] (0,3) -- (1,2) -- (2,1) -- (3,0);

    \foreach \P in {(0,2),(0,1),(1,1),(1,0),(2,0)}
        \node[ptPurple] at \P {};
    \foreach \P in {(0,3),(1,2),(2,1),(3,0)}
        \node[ptGreen] at \P {};

    \node[anchor=south east] at (-0.15,3.15) {$p$};
    \node[anchor=north west] at (3.12,0.05) {$q$};
\end{scope}

\begin{scope}[shift={(9.70,0)}]
    \draw[tri] (0,0) -- (0,3) -- (3,0) -- cycle;

    \fill[cyan!30,opacity=0.65] (0,1) -- (1,2) -- (1,1) -- (1,0) -- (0,0) -- cycle;
    \fill[cyan!30,opacity=0.65] (1,0) -- (2,1) -- (2,0) -- cycle;

    \draw[latticePurple]
        (0,1) -- (1,2) -- (1,1) -- (1,0) -- (2,1) -- (2,0);

    \draw[minusCyan] (0,3) -- (0,2) -- (0,1) -- (0,0) -- (1,0) -- (2,0) -- (3,0);

    \foreach \P in {(1,2),(1,1),(2,1)}
        \node[ptPurple] at \P {};
    \foreach \P in {(0,3),(0,2),(0,1),(0,0),(1,0),(2,0),(3,0)}
        \node[ptCyan] at \P {};

    \node[anchor=south east] at (-0.15,3.15) {$p$};
    \node[anchor=north west] at (3.12,0.05) {$q$};
\end{scope}

\end{tikzpicture}
\caption{A lattice path $\gamma$ in $\Delta_3$ and the corresponding decompositions into $\Delta_+$ and $\Delta_-$.} 
\label{lattice_p1}
\end{figure}    
\FloatBarrier


 Let $\gamma$ be any lattice path in a polygon $\Delta$. Let $\gamma(0)=p$ and $\gamma(n)=q$. Let $\alpha_+$ be the lattice path along $\partial \Delta$ that travels clockwise from $p$ to $q$, and $\alpha_-$ be the lattice path that goes counterclockwise from $p$ to $q$. Let $n^+=\alpha_+^{-1}(q)$ and $n^-=\alpha_-^{-1}(q)$. Split $\Delta$ into two polygons $\Delta_+$ and $\Delta_-$, where $\partial \Delta_+=\gamma \cup \alpha_+$ and $\partial \Delta_-=\gamma \cup \alpha_-$ (Figure \ref{lattice_p1}). Define $\mu_{\pm}(\alpha_{\pm})=1$ as the multiplicity of $\alpha_{\pm}$. The algorithm below, taken from \cite{MikCountingCurvesLatticePaths}, computes the multiplicity $\mu_+(\gamma)$ recursively. The algorithm that computes $\mu_-(\gamma)$ is identical by replacing $+$ with $-$ everywhere.

 \begin{itemize}
     \item[I.] Let $k$ be the smallest integer in $\{0,1, \dots, n\}$ such that $\gamma(k)$ is a locally convex vertex of $\Delta_+$.
     \begin{itemize}
         \item[(i)] Let $\gamma'$ be the lattice path of length $n-1$ such that $\gamma'(j)=\gamma(j)$ for $j<k$ and with $\gamma'(j)=\gamma(j+1)$ for $k\leq j \leq n-1$. If $n-1<n^+$, then $\mu_+(\gamma')=0$.
         \item[(ii)] Let $\gamma''$ be the lattice path of length $n$ such that $\gamma''(j)=\gamma(j)$ for $j\neq k$ and with $\gamma''(k)=\gamma(k-1)+\gamma(k+1)-\gamma(k))$. If $\gamma''(k)\notin \Delta$, then $\mu_+(\gamma'')=0$.
     \end{itemize}
     \item[II.] Define $\mu_+(\gamma)=2Area(T)\mu_+(\gamma')+\mu_+(\gamma'')$, where $T$ is the triangle with vertices $\gamma(k-1), \gamma(k), \gamma(k+1)$.     \item[III.] Repeat I and II for $\gamma'$ and $\gamma''$ until $\mu_+=0$ or until the lattice path converges to $\alpha_+$.
 \end{itemize}

\begin{definition} Let $\gamma:[0,n]\rightarrow \mathbb{R}^2$ be a $\lambda$-increasing lattice path in the lattice polygon $\Delta$. Then the \textbf{multiplicity of $\gamma$} is $\mu(\gamma)=\mu_+(\gamma)\mu_-(\gamma)$.
\end{definition}
 Figure \ref{p2_n03_mu+} displays the computation for $\mu_+(\gamma)$, $\gamma$ the lattice path in $\Delta_3$ in Figure \ref{lattice_p1}; Figure \ref{p2_n03_mu-} displays the computation for $\mu_-(\gamma)$. Together, they give $\mu(\gamma)=\mu_+(\gamma)\mu_-(\gamma)=1(2)=2$.
 
 The upshot is that the multiplicity of each $\lambda$-increasing lattice path in $\Delta$ contributes to the count of $N(g, \Delta)$, which when $g=0$ and $\Delta=d$ is the genus 0 Gromov-Witten invariant $N_d$ described in Sections~\ref{sec: GWinvtsP2}--\ref{sec: 4.5}. 
 This is summarized in the following two theorems.

 \begin{theorem}[Theorem 1 in \cite{MikEnumTropGeom}] \label{thm: 7.2} For a generic configuration $\mathcal{P}$ of $3d+g-1$ points in general position in $\mathbb{R}^2$, 
 $$N^{irr}_{trop}(g, \Delta)=N^{irr}(g,\Delta) \text{ and } N_{trop}(g, \Delta)= N(g, \Delta)$$

 Moreover, there exists a configuration $\mathcal{Q}\subset (\mathbb{C}^*)^2$ of $3d+g-1$ points in general position such that for every tropical curve $C$ of genus $g$, degree $\Delta$, passing through $\mathcal{P}$, we have $mult(C)$ distinct complex curves of genus $g$, degree $\Delta$, passing through $\mathcal{Q}$.
     
 \end{theorem}

\begin{theorem}[Theorem 2 in \cite{MikEnumTropGeom}] \label{thm: 7.3} $N_{trop}(g, \Delta)$ is the number, counting multiplicity, of $\lambda$-increasing lattice paths $\gamma: [0,n]\rightarrow \Delta$ connecting $p$ and $q$. Moreover, there exists a configuration $\mathcal{P}\subset \mathbb{R}^2$ of $3d+g-1$ points in tropical general position such that each $\lambda$-increasing path encodes the numbers of tropical curves of genus $g$, degree $\Delta$, passing through $\mathcal{P}$. Distinct curves enumerate distinct paths.

\end{theorem}

Thus, to compute $N_3$ for $\mathbb{P}^2$, we can compute $N^{irr}_{trop}(g, \Delta_3)$. The lattice path $\gamma$ in figures \ref{lattice_p1}, \ref{p2_n03_mu+}, and \ref{p2_n03_mu-} contributes to this count, enumerating $2$ tropical curves passing through $8$ points in tropical general position in $\mathbb{R}^2$ by Theorem~\ref{thm: 7.3}. Therefore, by Theorem~\ref{thm: 7.3}, this lattice path enumerates $2$ complex curves passing through $8$ points in general position in $\mathbb{P}^2$. Computing the multiplicity of the other $\lambda$-increasing lattice paths in $\Delta_3$, we would get a total count of $12$ curves. This is the familiar number $N_3$ for $\mathbb{P}^2$.

More recently, there have been improvements in algorithms that compute generating series for descendant GW-invariants and Hurwitz numbers for elliptic curves using their tropical counterparts, which the interested reader can find, for example, in \cite{aga2023algorithmsgromovwitteninvariantselliptic}.

\FloatBarrier
\begin{figure}[ht]
\centering
%
\begingroup
\definecolor{muppgreen}{RGB}{99,184,75}
\definecolor{mupplightgreen}{RGB}{182,221,172}
\definecolor{mupppurple}{RGB}{164,4,198}
\definecolor{muppred}{RGB}{236,54,36}
\definecolor{mupppink}{RGB}{254,6,212}
\definecolor{mupphighlight}{RGB}{254,253,148}
\definecolor{muppgray}{RGB}{105,105,105}

\tikzset{
  muppaxis/.style={draw=muppgray,line width=1.05pt,line cap=round,line join=round},
  muppgreenedge/.style={draw=muppgreen,line width=3.3pt,line cap=round,line join=round},
  mupppurpedge/.style={draw=mupppurple,line width=2.35pt,line cap=round,line join=round},
  mupparrow/.style={draw=black,line width=1.4pt,-{Stealth[length=6pt,width=5pt]},line cap=round,line join=round},
  mupprednote/.style={text=muppred,font=\fontsize{15}{17}\selectfont},
  muppgreennote/.style={text=muppgreen,font=\fontsize{18}{20}\selectfont},
  muppblacknote/.style={text=black,font=\fontsize{12}{14}\selectfont},
}

\def\muppPt#1#2#3#4#5{({#1+#3*(#4)},{#2+#3*(#5)})}
\def\muppGdot#1{\fill[muppgreen] #1 circle[radius=4.2pt];}
\def\muppPdot#1{\fill[mupppurple] #1 circle[radius=3.7pt];}
\def\muppReddot#1{\fill[muppred] #1 circle[radius=4.0pt];}
\def\muppRedcircle#1{\draw[muppred,line width=2.2pt] #1 circle[radius=7.0pt];}

\def\muppBase#1#2#3{%
  \draw[muppaxis] \muppPt{#1}{#2}{#3}{0}{0}--\muppPt{#1}{#2}{#3}{0}{3}--\muppPt{#1}{#2}{#3}{3}{3};
  \draw[muppgreenedge] \muppPt{#1}{#2}{#3}{0}{0}--\muppPt{#1}{#2}{#3}{3}{3};
  \foreach \i in {0,1,2,3}{\muppGdot{\muppPt{#1}{#2}{#3}{\i}{\i}}}
  \node[black,anchor=south east,font=\fontsize{16}{18}\selectfont,inner sep=0pt] at \muppPt{#1}{#2}{#3}{-.10}{-.08} {$p$};
  \node[black,anchor=north west,font=\fontsize{16}{18}\selectfont,inner sep=0pt] at \muppPt{#1}{#2}{#3}{3.08}{3.05} {$q$};
}

\def\muppPathgamma#1#2#3#4{
  \fill[mupplightgreen,opacity=.82]
    \muppPt{#1}{#2}{#3}{0}{0}--\muppPt{#1}{#2}{#3}{0}{1}--\muppPt{#1}{#2}{#3}{1}{1}--cycle;
  \fill[mupplightgreen,opacity=.82]
    \muppPt{#1}{#2}{#3}{1}{1}--\muppPt{#1}{#2}{#3}{1}{3}--\muppPt{#1}{#2}{#3}{2}{2}--cycle;
  \fill[mupplightgreen,opacity=.82]
    \muppPt{#1}{#2}{#3}{2}{2}--\muppPt{#1}{#2}{#3}{2}{3}--\muppPt{#1}{#2}{#3}{3}{3}--cycle;
  \fill[mupppurple,opacity=.30]
    \muppPt{#1}{#2}{#3}{0}{1}--\muppPt{#1}{#2}{#3}{0}{2}--\muppPt{#1}{#2}{#3}{1}{1}--cycle;
  \muppBase{#1}{#2}{#3}
  \draw[mupppurpedge] \muppPt{#1}{#2}{#3}{0}{0}--\muppPt{#1}{#2}{#3}{0}{2};
  \draw[mupppurpedge] \muppPt{#1}{#2}{#3}{0}{1}--\muppPt{#1}{#2}{#3}{1}{1};
  \draw[mupppurpedge] \muppPt{#1}{#2}{#3}{0}{2}--\muppPt{#1}{#2}{#3}{1}{1};
  \draw[mupppurpedge] \muppPt{#1}{#2}{#3}{1}{1}--\muppPt{#1}{#2}{#3}{1}{3};
  \draw[mupppurpedge] \muppPt{#1}{#2}{#3}{1}{3}--\muppPt{#1}{#2}{#3}{2}{2};
  \draw[mupppurpedge] \muppPt{#1}{#2}{#3}{2}{2}--\muppPt{#1}{#2}{#3}{2}{3};
  \draw[mupppurpedge] \muppPt{#1}{#2}{#3}{2}{3}--\muppPt{#1}{#2}{#3}{3}{3};
  \foreach \a/\b in {0/1,0/2,1/2,1/3,2/3}{\muppPdot{\muppPt{#1}{#2}{#3}{\a}{\b}}}
  \muppRedcircle{#4}
}

\def\muppPathgammaprime#1#2#3#4{
  \fill[mupplightgreen,opacity=.82]
    \muppPt{#1}{#2}{#3}{1}{1}--\muppPt{#1}{#2}{#3}{1}{3}--\muppPt{#1}{#2}{#3}{2}{2}--cycle;
  \fill[mupplightgreen,opacity=.82]
    \muppPt{#1}{#2}{#3}{2}{2}--\muppPt{#1}{#2}{#3}{2}{3}--\muppPt{#1}{#2}{#3}{3}{3}--cycle;
  \fill[mupppurple,opacity=.30]
    \muppPt{#1}{#2}{#3}{0}{0}--\muppPt{#1}{#2}{#3}{0}{1}--\muppPt{#1}{#2}{#3}{1}{1}--cycle;
  \muppBase{#1}{#2}{#3}
  \draw[mupppurpedge] \muppPt{#1}{#2}{#3}{0}{0}--\muppPt{#1}{#2}{#3}{0}{1};
  \draw[mupppurpedge] \muppPt{#1}{#2}{#3}{0}{1}--\muppPt{#1}{#2}{#3}{1}{1};
  \draw[mupppurpedge] \muppPt{#1}{#2}{#3}{1}{1}--\muppPt{#1}{#2}{#3}{1}{3};
  \draw[mupppurpedge] \muppPt{#1}{#2}{#3}{1}{3}--\muppPt{#1}{#2}{#3}{2}{2};
  \draw[mupppurpedge] \muppPt{#1}{#2}{#3}{2}{2}--\muppPt{#1}{#2}{#3}{2}{3};
  \draw[mupppurpedge] \muppPt{#1}{#2}{#3}{2}{3}--\muppPt{#1}{#2}{#3}{3}{3};
  \foreach \a/\b in {0/1,1/2,1/3,2/3}{\muppPdot{\muppPt{#1}{#2}{#3}{\a}{\b}}}
  \muppRedcircle{#4}
}

\def\muppPathdouble#1#2#3{
  \fill[mupplightgreen,opacity=.82]
    \muppPt{#1}{#2}{#3}{0}{0}--\muppPt{#1}{#2}{#3}{0}{1}--\muppPt{#1}{#2}{#3}{1}{1}--cycle;
  \fill[mupplightgreen,opacity=.82]
    \muppPt{#1}{#2}{#3}{1}{1}--\muppPt{#1}{#2}{#3}{1}{3}--\muppPt{#1}{#2}{#3}{2}{2}--cycle;
  \fill[mupplightgreen,opacity=.82]
    \muppPt{#1}{#2}{#3}{2}{2}--\muppPt{#1}{#2}{#3}{2}{3}--\muppPt{#1}{#2}{#3}{3}{3}--cycle;
  \muppBase{#1}{#2}{#3}
  \draw[mupppurpedge] \muppPt{#1}{#2}{#3}{0}{0}--\muppPt{#1}{#2}{#3}{0}{1};
  \draw[mupppurpedge] \muppPt{#1}{#2}{#3}{0}{1}--\muppPt{#1}{#2}{#3}{1}{1};
  \draw[mupppurpedge] \muppPt{#1}{#2}{#3}{1}{1}--\muppPt{#1}{#2}{#3}{1}{3};
  \draw[mupppurpedge] \muppPt{#1}{#2}{#3}{1}{3}--\muppPt{#1}{#2}{#3}{2}{2};
  \draw[mupppurpedge] \muppPt{#1}{#2}{#3}{2}{2}--\muppPt{#1}{#2}{#3}{2}{3};
  \draw[mupppurpedge] \muppPt{#1}{#2}{#3}{2}{3}--\muppPt{#1}{#2}{#3}{3}{3};
  \foreach \a/\b in {0/1,1/2,1/3,2/3}{\muppPdot{\muppPt{#1}{#2}{#3}{\a}{\b}}}
}

\def\muppPathafterone#1#2#3#4{
  \fill[mupplightgreen,opacity=.82]
    \muppPt{#1}{#2}{#3}{1}{1}--\muppPt{#1}{#2}{#3}{1}{2}--\muppPt{#1}{#2}{#3}{2}{2}--cycle;
  \fill[mupppurple,opacity=.30]
    \muppPt{#1}{#2}{#3}{1}{2}--\muppPt{#1}{#2}{#3}{1}{3}--\muppPt{#1}{#2}{#3}{2}{2}--cycle;
  \fill[mupplightgreen,opacity=.82]
    \muppPt{#1}{#2}{#3}{2}{2}--\muppPt{#1}{#2}{#3}{2}{3}--\muppPt{#1}{#2}{#3}{3}{3}--cycle;
  \muppBase{#1}{#2}{#3}
  \draw[mupppurpedge] \muppPt{#1}{#2}{#3}{0}{0}--\muppPt{#1}{#2}{#3}{0}{1.55};
  \draw[mupppurpedge] \muppPt{#1}{#2}{#3}{1}{1}--\muppPt{#1}{#2}{#3}{1}{3};
  \draw[mupppurpedge] \muppPt{#1}{#2}{#3}{1}{3}--\muppPt{#1}{#2}{#3}{2}{2};
  \draw[mupppurpedge] \muppPt{#1}{#2}{#3}{2}{2}--\muppPt{#1}{#2}{#3}{2}{3};
  \draw[mupppurpedge] \muppPt{#1}{#2}{#3}{2}{3}--\muppPt{#1}{#2}{#3}{3}{3};
  \foreach \a/\b in {1/2,1/3,2/3}{\muppPdot{\muppPt{#1}{#2}{#3}{\a}{\b}}}
  \muppRedcircle{#4}
}

\def\muppPathafterthree#1#2#3#4{
  \fill[mupppurple,opacity=.30]
    \muppPt{#1}{#2}{#3}{1}{1}--\muppPt{#1}{#2}{#3}{1}{2}--\muppPt{#1}{#2}{#3}{2}{2}--cycle;
  \fill[mupplightgreen,opacity=.82]
    \muppPt{#1}{#2}{#3}{2}{2}--\muppPt{#1}{#2}{#3}{2}{3}--\muppPt{#1}{#2}{#3}{3}{3}--cycle;
  \muppBase{#1}{#2}{#3}
  \draw[mupppurpedge] \muppPt{#1}{#2}{#3}{0}{0}--\muppPt{#1}{#2}{#3}{0}{1.55};
  \draw[mupppurpedge] \muppPt{#1}{#2}{#3}{1}{1}--\muppPt{#1}{#2}{#3}{1}{2};
  \draw[mupppurpedge] \muppPt{#1}{#2}{#3}{1}{2}--\muppPt{#1}{#2}{#3}{2}{2};
  \draw[mupppurpedge] \muppPt{#1}{#2}{#3}{2}{2}--\muppPt{#1}{#2}{#3}{2}{3};
  \draw[mupppurpedge] \muppPt{#1}{#2}{#3}{2}{3}--\muppPt{#1}{#2}{#3}{3}{3};
  \foreach \a/\b in {1/2,2/3}{\muppPdot{\muppPt{#1}{#2}{#3}{\a}{\b}}}
  \muppRedcircle{#4}
}

\def\muppPathprealpha#1#2#3#4{
  \fill[mupppurple,opacity=.30]
    \muppPt{#1}{#2}{#3}{2}{2}--\muppPt{#1}{#2}{#3}{2}{3}--\muppPt{#1}{#2}{#3}{3}{3}--cycle;
  \muppBase{#1}{#2}{#3}
  \draw[mupppurpedge] \muppPt{#1}{#2}{#3}{0}{0}--\muppPt{#1}{#2}{#3}{0}{1.45};
  \draw[mupppurpedge] \muppPt{#1}{#2}{#3}{2}{2}--\muppPt{#1}{#2}{#3}{2}{3};
  \draw[mupppurpedge] \muppPt{#1}{#2}{#3}{2}{3}--\muppPt{#1}{#2}{#3}{3}{3};
  \muppPdot{\muppPt{#1}{#2}{#3}{2}{3}}
  \muppRedcircle{#4}
}

\def\muppPathalpha#1#2#3{
  \muppBase{#1}{#2}{#3}
  \draw[mupppurpedge] \muppPt{#1}{#2}{#3}{0}{0}--\muppPt{#1}{#2}{#3}{0}{1.45};
  \draw[mupppurpedge] \muppPt{#1}{#2}{#3}{2.55}{3}--\muppPt{#1}{#2}{#3}{3}{3};
}

\noindent\resizebox{0.93\linewidth}{!}{%
\begin{tikzpicture}[x=1bp,y=-1bp]
  \path[use as bounding box] (-70,0) rectangle (760,780);
  \fill[white] (-70,0) rectangle (760,780);

  \muppPathgamma{326}{40}{22}{\muppPt{326}{40}{22}{0}{2}}
  \node[mupprednote,anchor=east] at (322,83) {$\gamma(2)$};
  \node[text=mupppink,font=\bfseries\fontsize{13}{15}\selectfont]
    at \muppPt{326}{40}{22}{.31}{1.33} {$T$};
  \node[muppgreennote,anchor=west] at (362,60) {$\Delta^{+}$};
  \node[text=mupppink,font=\fontsize{15}{17}\selectfont,anchor=west] (muppAreaT)
    at (410,70) {$\mathrm{Area}(T)=\frac12$};
  \coordinate (muppTtarget) at \muppPt{326}{40}{22}{.76}{1.22};
  \draw[draw=mupppink,line width=.9pt,-{Stealth[length=4.5pt,width=3.5pt]}]
    (muppAreaT.west) to[out=185,in=8] (muppTtarget);
  \node[muppgreennote,anchor=west] at (324,129) {$\alpha^{+}:[0,3]\to\partial\Delta_{3}$};
  \node[muppgreennote,anchor=west] at (331,160) {$n^{+}=3$};
  \node[mupprednote,anchor=west,font=\fontsize{15}{17}\selectfont] at (337,192) {$k=2$};

  \draw[mupparrow] (307,139) .. controls (299,154) and (292,159) .. (285,163);
  \draw[mupparrow] (401,139) .. controls (409,153) and (417,158) .. (425,163);

  \muppPathgammaprime{230}{162}{24}{\muppPt{230}{162}{24}{0}{1}}
  \node[black,font=\fontsize{15}{17}\selectfont,anchor=east] at (206,202) {$\gamma'$};
  \node[mupprednote,anchor=west,font=\fontsize{15}{17}\selectfont] at (249,249) {$k=1$};
  \draw[mupparrow] (231,239) .. controls (225,252) and (220,255) .. (215,262);
  \draw[mupparrow] (287,239) .. controls (293,250) and (300,255) .. (307,262);

  \muppPathdouble{456}{161}{24}
  \node[black,font=\fontsize{15}{17}\selectfont,anchor=east] at (433,203) {$\gamma''$};
  \node[muppblacknote,anchor=west,align=left] at (420,258)
    {$\begin{aligned}
       \gamma''(2)\;&{}=(0,2)+(1,2)-(0,1)\\[-1pt]
       &{}=(1,3)\ {\color{muppred}\times}
      \end{aligned}$};
  \node[fill=mupphighlight,rounded corners=3pt,opacity=.75,text opacity=1,text=muppred,font=\fontsize{17}{19}\selectfont,anchor=west,inner xsep=3pt,inner ysep=1pt] at (438,316) {$\mu_{+}(\gamma'')=0.$};

  \node[muppblacknote,anchor=west,align=left] at (286,286)
    {$\begin{aligned}
       (\gamma')''(1)\;&{}=(0,3)+(1,2)-(0,2)\\[-1pt]
       &{}=(1,3)\ {\color{muppred}\times}
      \end{aligned}$};
  \node[fill=mupphighlight,rounded corners=3pt,opacity=.75,text opacity=1,text=muppred,font=\fontsize{17}{19}\selectfont,anchor=west,inner xsep=3pt,inner ysep=1pt] at (292,330) {$\mu_{+}\big((\gamma')''\big)=0$};

  \muppPathafterone{165}{278}{22}{\muppPt{165}{278}{22}{1}{3}}
  \node[black,font=\fontsize{15}{17}\selectfont,anchor=east] at (145,315) {$(\gamma')'$};
  \node[mupprednote,anchor=west,font=\fontsize{15}{17}\selectfont] at (175,363) {$k=3$};

  \draw[mupparrow] (160,351) .. controls (147,367) and (132,378) .. (116,392);
  \draw[mupparrow] (225,350) .. controls (245,362) and (270,372) .. (298,388);

  \node[muppblacknote,anchor=west,align=left] at (308,398)
    {$\begin{aligned}
       ((\gamma')')''(3)\;&{}=(1,1)+(2,1)-(0,1)\\[-1pt]
       &{}=(3,1)\ {\color{muppred}\times}
      \end{aligned}$};
  \node[text=muppred,font=\fontsize{15}{17}\selectfont,anchor=west] at (350,443) {$\mu_{+}\big(((\gamma')')''\big)=0$};

  \muppPathafterthree{111}{391}{23}{\muppPt{111}{391}{23}{1}{2}}
  \node[black,font=\fontsize{15}{17}\selectfont,anchor=east] at (92,432) {$((\gamma')')'$};
  \node[mupprednote,anchor=west,font=\fontsize{15}{17}\selectfont] at (120,477) {$k=2$};

  \draw[mupparrow] (110,466) .. controls (98,486) and (84,501) .. (68,517);
  \draw[mupparrow] (167,466) .. controls (187,482) and (212,495) .. (240,512);

  \node[muppblacknote,anchor=west,align=left] at (248,520)
    {$\begin{aligned}
       &(1,2)+(2,1)-(1,1)\\[-1pt]
       &\qquad=(2,2)\ {\color{muppred}\times}
      \end{aligned}$};

  \muppPathprealpha{63}{520}{24}{\muppPt{63}{520}{24}{2}{3}}
  \node[black,font=\fontsize{15}{17}\selectfont,anchor=east] at (45,562) {$(((\gamma')')')'$};
  \node[mupprednote,anchor=west,font=\fontsize{15}{17}\selectfont] at (80,610) {$k=3$};

  \draw[mupparrow] (70,595) .. controls (58,615) and (49,633) .. (42,650);
  \draw[mupparrow] (132,595) .. controls (170,611) and (220,623) .. (272,636);

  \node[muppblacknote,anchor=west] at (282,646)
    {$(2,1)+(3,0)-(2,0)=(3,1)\ {\color{muppred}\times}$};

  \muppPathalpha{42}{654}{26}
  \node[muppgreennote,anchor=west] at (160,723)
    {$\alpha^{+}:[0,3]\to\partial\Delta$};
  \node[
    fill=mupphighlight,
    rounded corners=4pt,
    opacity=.75,
    text opacity=1,
    text=muppgreen,
    font=\fontsize{17}{19}\selectfont,
    anchor=west,
    inner xsep=4pt,
    inner ysep=2pt
  ] at (160,752) {$\mu_{+}(\alpha^{+})=1$};
\end{tikzpicture}
}%
\endgroup

\caption{Computation of $\mu_+(\gamma)$.}
\label{p2_n03_mu+}
\end{figure}
\FloatBarrier

\FloatBarrier
\begin{figure}[ht]
\centering
\begingroup
\definecolor{mmcyan}{RGB}{22,190,211}
\definecolor{mmlightcyan}{RGB}{164,232,241}
\definecolor{mmpurple}{RGB}{154,31,190}
\definecolor{mmlightpurple}{RGB}{177,153,229}
\definecolor{mmred}{RGB}{239,72,55}
\definecolor{mmhighlight}{RGB}{255,248,134}
\definecolor{mmgray}{RGB}{72,72,72}

\tikzset{
  mmaxis/.style={draw=mmcyan,line width=2.4pt,line cap=round,line join=round},
  mmdiag/.style={draw=mmgray!78,line width=1.15pt,line cap=round},
  mmpath/.style={draw=black,line width=1.45pt,line cap=round,line join=round},
  mmpurpleedge/.style={draw=mmpurple,line width=1.65pt,line cap=round,line join=round},
  mmcyanedge/.style={draw=mmcyan!80!black,line width=1.25pt,line cap=round,line join=round},
  mmarrow/.style={draw=black,line width=1.35pt,-{Stealth[length=7pt,width=5.5pt]},line cap=round},
  mmtitle/.style={text=black,font=\fontsize{19}{21}\selectfont},
  mmrednote/.style={text=mmred,font=\fontsize{17}{19}\selectfont},
  mmcyannote/.style={text=mmcyan!85!black,font=\fontsize{17}{19}\selectfont},
  mmeq/.style={text=black,font=\fontsize{16.5}{19.5}\selectfont,align=left,fill=white,inner sep=2pt},
  mmbig/.style={text=black,font=\fontsize{17}{20}\selectfont},
}

\def\mmCdot#1#2{\fill[mmcyan] (#1,#2) circle[radius=4.0pt];}
\def\mmBdot#1#2{\fill[black] (#1,#2) circle[radius=3.75pt];}
\def\mmPdot#1#2{\fill[mmpurple] (#1,#2) circle[radius=3.75pt];}
\def\mmSelect#1#2{\draw[mmred,line width=2.2pt] (#1,#2) circle[radius=9.0pt];}

\def\mmFrame{%
  \draw[mmaxis] (0,0)--(0,60)--(60,60);
  \draw[mmdiag] (0,0)--(60,60);
  \node[mmtitle,anchor=south east,inner sep=1pt] at (-3,-1) {$p$};
  \node[mmtitle,anchor=north west,inner sep=1pt] at (63,60) {$q$};
}
\def\mmAllCyanBoundary{%
  \foreach \y in {0,20,40,60}{\mmCdot{0}{\y}}
  \foreach \x in {20,40,60}{\mmCdot{\x}{60}}
}
\def\mmAllBlackBoundary{%
  \foreach \y in {0,20,40,60}{\mmBdot{0}{\y}}
  \foreach \x in {20,40,60}{\mmBdot{\x}{60}}
}

\def\mmRoot{%
  \fill[mmlightcyan] (0,40)--(0,60)--(20,60)--(20,40)--cycle;
  \fill[mmlightcyan] (20,60)--(40,40)--(40,60)--cycle;
  \fill[mmlightpurple] (0,40)--(20,20)--(20,40)--cycle;
  \mmFrame
  \draw[mmpurpleedge] (0,0)--(0,40)--(20,20)--(20,40)--(20,60)--(40,40)--(40,60);
  \draw[mmpurpleedge] (0,40)--(20,40);
  \mmAllCyanBoundary
  \mmPdot{20}{20}\mmPdot{20}{40}\mmPdot{40}{40}
  \mmSelect{20}{20}
}

\def\mmGammaPrime{%
  \fill[mmlightcyan] (0,40)--(0,60)--(20,60)--cycle;
  \fill[mmlightcyan] (20,60)--(40,40)--(40,60)--cycle;
  \fill[mmlightpurple] (0,40)--(20,40)--(20,60)--cycle;
  \mmFrame
  \draw[mmpath] (0,0)--(0,40)--(20,40)--(20,60)--(40,40)--(40,60)--(60,60);
  \draw[mmcyanedge] (0,40)--(20,60);
  \draw[mmpurpleedge] (0,40)--(20,40)--(20,60);
  \draw[mmpurpleedge] (20,60)--(40,40)--(40,60);
  \mmCdot{0}{60}
  \foreach \y in {0,20,40}{\mmBdot{0}{\y}}
  \foreach \x in {20,40,60}{\mmBdot{\x}{60}}
  \mmBdot{20}{40}\mmBdot{40}{40}
  \mmSelect{20}{40}
}

\def\mmGammaPrimePrimeLeft{%
  \fill[mmlightcyan] (0,40)--(0,60)--(20,60)--cycle;
  \fill[mmlightpurple] (20,60)--(40,40)--(40,60)--cycle;
  \mmFrame
  \draw[mmpath] (0,0)--(0,40)--(20,60);
  \draw[mmpurpleedge] (20,60)--(40,40)--(40,60);
  \draw[mmpath] (40,60)--(60,60);
  \mmCdot{0}{60}
  \foreach \y in {0,20,40}{\mmBdot{0}{\y}}
  \foreach \x in {20,40,60}{\mmBdot{\x}{60}}
  \mmBdot{40}{40}
  \mmSelect{40}{40}
}

\def\mmGammaPrimePrimeTerminal{%
  \fill[mmlightcyan] (0,40)--(0,60)--(20,60)--cycle;
  \mmFrame
  \draw[mmpath] (0,0)--(0,40)--(20,60)--(60,60);
  \mmCdot{0}{60}
  \foreach \y in {0,20,40}{\mmBdot{0}{\y}}
  \foreach \x in {20,40,60}{\mmBdot{\x}{60}}
}

\def\mmGammaPrimeDouble{%
  \fill[mmlightpurple] (20,60)--(40,40)--(40,60)--cycle;
  \mmFrame
  \draw[mmpath] (0,0)--(0,60)--(60,60);
  \draw[mmpurpleedge] (20,60)--(40,40)--(40,60);
  \mmAllBlackBoundary
  \mmBdot{40}{40}
  \mmSelect{40}{40}
}

\def\mmGammaDouble{%
  \fill[mmlightpurple] (0,60)--(20,40)--(20,60)--cycle;
  \fill[mmlightcyan] (20,60)--(40,40)--(40,60)--cycle;
  \mmFrame
  \draw[mmpurpleedge] (0,0)--(0,60)--(20,40)--(20,60)--(40,40)--(40,60);
  \mmAllCyanBoundary
  \mmPdot{20}{40}\mmPdot{40}{40}
  \mmSelect{20}{40}
}

\def\mmGammaDoublePrime{%
  \fill[mmlightpurple] (20,60)--(40,40)--(40,60)--cycle;
  \mmFrame
  \draw[mmpurpleedge] (0,0)--(0,60);
  \draw[mmpurpleedge] (20,60)--(40,40)--(40,60);
  \mmAllCyanBoundary
  \mmPdot{40}{40}
  \mmSelect{40}{40}
}

\def\mmAlphaMinus{%
  \mmFrame
  \mmAllCyanBoundary
}

\resizebox{0.98\textwidth}{!}{%
\begin{tikzpicture}[x=1pt,y=-1pt]
  \path[use as bounding box] (0,0) rectangle (1160,900);

  \begin{scope}[shift={(525,35)}]\mmRoot\end{scope}
  \node[mmrednote,anchor=west] at (552,50) {$\gamma(3)$};
  \node[mmcyannote,anchor=west] at (608,78) {$\Delta^{-}$};
  \node[mmcyannote,anchor=west] at (493,116)
    {$\alpha^{-}:[0,6]\longrightarrow 2\Delta_{3}$};
  \node[mmcyannote,anchor=west] at (535,145) {$n^{-}=6$};
  \node[mmrednote,anchor=west] at (540,174) {$k=3$};

  \draw[mmarrow] (500,190) .. controls (455,217) and (392,226) .. (326,247);
  \draw[mmarrow] (615,190) .. controls (664,217) and (722,226) .. (790,247);

  \begin{scope}[shift={(245,245)}]\mmGammaPrime\end{scope}
  \node[mmtitle,anchor=east] at (232,265) {$\gamma'$};
  \node[mmrednote,anchor=west] at (266,326) {$k=3$};

  \begin{scope}[shift={(800,245)}]\mmGammaDouble\end{scope}
  \node[mmtitle,anchor=west] at (875,265) {$\gamma''$};
  \node[mmeq,anchor=west] at (748,343)
    {$\begin{aligned}
      \gamma''(3)&=(0,1)+(1,1)-(1,2)\\[-1pt]
                  &=(0,0)
    \end{aligned}$};
  \node[mmrednote,anchor=west] at (827,385) {$k=4$};

  \draw[mmarrow] (260,344) .. controls (225,376) and (179,393) .. (137,421);
  \draw[mmarrow] (306,344) .. controls (350,372) and (404,393) .. (458,421);

  \node[mmeq,anchor=west] at (355,346)
    {$\begin{aligned}
      (\gamma')''(3)&=(0,1)+(1,0)-(1,1)\\[-1pt]
                     &=(0,0)
    \end{aligned}$};

  \begin{scope}[shift={(95,430)}]\mmGammaPrimePrimeLeft\end{scope}
  \node[mmtitle,anchor=east] at (84,454) {$\bigl(\gamma'\bigr)'$};
  \node[mmrednote,anchor=west] at (120,511) {$k=4$};

  \begin{scope}[shift={(455,430)}]\mmGammaPrimeDouble\end{scope}
  \node[mmtitle,anchor=west] at (530,450) {$\bigl(\gamma'\bigr)''$};
  \node[mmrednote,anchor=west] at (479,511) {$k=7$};

  \draw[mmarrow] (805,411) .. controls (775,433) and (740,444) .. (707,466);
  \draw[mmarrow] (857,411) .. controls (896,433) and (934,444) .. (977,466);

  \node[mmeq,anchor=west] at (917,455)
    {$\begin{aligned}
      &(0,0)+(1,0)-(1,1)\\[-1pt]
      &\qquad=(0,-1)\;\textcolor{mmred}{\times}
    \end{aligned}$};

  \begin{scope}[shift={(680,475)}]\mmGammaDoublePrime\end{scope}
  \node[mmtitle,anchor=west] at (753,493) {$\bigl(\gamma''\bigr)'$};
  \node[mmrednote,anchor=west] at (704,556) {$k=5$};

  \draw[mmarrow] (108,531) .. controls (91,562) and (77,584) .. (65,612);
  \draw[mmarrow] (151,531) .. controls (199,558) and (255,579) .. (307,607);
  \node[mmeq,anchor=west] at (228,548)
    {$\begin{aligned}
       &(1,0)+(2,0)-(2,1)\\[-1pt]
       &\qquad=(1,-1)\;\textcolor{mmred}{\times}
     \end{aligned}$};

  \draw[mmarrow] (472,531) .. controls (450,563) and (428,584) .. (408,612);
  \draw[mmarrow] (519,531) .. controls (555,555) and (589,574) .. (622,600);
  \node[mmeq,anchor=west] at (591,590)
    {$(1,-1)\;\textcolor{mmred}{\times}$};

  \draw[mmarrow] (693,576) .. controls (675,604) and (659,624) .. (645,650);
  \draw[mmarrow] (737,576) .. controls (778,600) and (820,620) .. (866,646);
  \node[mmeq,anchor=west] at (890,655)
    {$\begin{aligned}
       &(1,0)+(2,0)-(2,1)\\[-1pt]
       &\qquad=(1,-1)\;\textcolor{mmred}{\times}
     \end{aligned}$};

  \begin{scope}[shift={(35,625)}]\mmGammaPrimePrimeTerminal\end{scope}
  \node[mmtitle,anchor=east] at (24,650) {$\bigl((\gamma')'\bigr)'$};
  \node[mmbig,anchor=west] at (75,702) {$n=5<n^{-}$};
  \node[mmbig,anchor=west,fill=mmhighlight,rounded corners=3pt,inner xsep=4pt,inner ysep=2pt]
    at (18,744) {$\Rightarrow\mu_{-}\!\left(\bigl((\gamma')'\bigr)'\right)=0$};
  \node[mmbig,anchor=west,fill=mmhighlight,rounded corners=3pt,inner xsep=4pt,inner ysep=2pt]
    at (22,782) {$\Rightarrow\mu_{-}\!\left((\gamma')'\right)=0$};

  \begin{scope}[shift={(385,625)}]\mmAlphaMinus\end{scope}
  \node[mmbig,anchor=west] at (350,708)
    {$\bigl((\gamma')''\bigr)'=\alpha^{-}$};
  \node[mmbig,anchor=west] at (396,744) {$\mu_{-}(\alpha^{-})=1$};
  \node[mmbig,anchor=west] at (337,780)
    {$\Rightarrow\mu_{-}\!\left(\bigl((\gamma')''\bigr)'\right)=1$};
  \node[mmbig,anchor=west,fill=mmhighlight,rounded corners=3pt,inner xsep=4pt,inner ysep=2pt]
    at (349,818) {$\Rightarrow\mu_{-}\!\left((\gamma')''\right)=1$};

  \begin{scope}[shift={(620,670)}]\mmAlphaMinus\end{scope}
  \node[mmcyannote,anchor=west] at (705,731) {$\alpha^{-}$};
  \node[mmbig,anchor=west,fill=mmhighlight,rounded corners=3pt,inner xsep=4pt,inner ysep=2pt]
    at (704,767) {$\mu_{-}(\alpha^{-})=1$};

\end{tikzpicture}%
}
\endgroup
\caption{Computation of $\mu_-(\gamma)$.}
\label{p2_n03_mu-}
\end{figure}
\FloatBarrier

\addtocontents{toc}{\protect\setcounter{tocdepth}{-1}}
\section*{Acknowledgments}

The first author is grateful to Dave Auckly, Lino Amorim, Alexander Givental, Mohammed Abouzaid, Jesse Wolfson, Chris Woodward, Edray Goins, Nick Sheridan, Weihong Xu, and Heather Lee for helpful discussions. 

\section*{AI acknowledgements}
ChatGPT 5.6 Pro was used to re-create images in TeX.

\printbibliography
\addtocontents{toc}{\protect\setcounter{tocdepth}{2}} 

\end{document}